\documentclass[a4paper]{article}

\usepackage{
amsmath,
amsthm,
amscd,
amssymb,
stmaryrd,
wasysym
}
\usepackage{comment}
\usepackage{tikz}
\usetikzlibrary{positioning,arrows}
\usetikzlibrary{decorations.pathreplacing}
\usetikzlibrary{decorations.pathmorphing}
\usetikzlibrary{patterns}
\usepackage{xstring}

\usepackage{pgfplots}
\usepgfplotslibrary{fillbetween}
\pgfplotsset{compat=1.18}

\usepackage{multicol}

\usepackage{ifthen}

\usepackage{subcaption}

\newtheorem{lemma}{Lemma}
\newtheorem{definition}{Definition}
\newtheorem{corollary}{Corollary}
\newtheorem{proposition}{Proposition}
\newtheorem{example}{Example}
\newtheorem{theorem}{Theorem}

\newtheorem{question}{Question}
\newtheorem{remark}{Remark}

\newcommand\xqed[1]{%
  \leavevmode\unskip\penalty9999 \hbox{}\nobreak\hfill
  \quad\hbox{#1}}
\newcommand\qee{\xqed{$\fullmoon$}}

\newcommand{\R}{\mathbb{R}}
\newcommand{\Q}{\mathbb{Q}}
\newcommand{\Z}{\mathbb{Z}}

\newcommand{\N}{\mathbb{N}}

\tikzset{snake it/.style={decorate, decoration={snake, segment length=1mm, amplitude=0.2mm}}}

\newcommand{\tle}[1]{
    \ifthenelse{\equal{#1}{ }}{
    }{}
    \ifthenelse{\equal{#1}{n}}{
        \draw (0.3,0.25) -- (0.5,0.75) -- (0.7,0.25);
    }{}
    \ifthenelse{\equal{#1}{s}}{
        \draw (0.3,0.75) -- (0.5,0.25) -- (0.7,0.75);
    }{}
    \ifthenelse{\equal{#1}{x}}{
        \draw (0.3,0.75) -- (0.5,0.25) -- (0.7,0.75);
        \draw (0.3,0.25) -- (0.5,0.75) -- (0.7,0.25);
    }{}
}

\newcommand{\dom}{\mathrm{dom}}
\newcommand{\REC}{\textsc{Rec}}

\newcommand{\cshape}{\mathrm{cshape}}

\newcommand{\lang}{\mathcal{L}}
\newcommand{\vortex}[1]{%
  \tikz[scale=#1]{
    \draw[fill=gray] (0,0) rectangle (1,1);
    \fill[black!15] (0.1,0.1) rectangle (0.9,0.9);
    \draw (0,0) arc (180:90:0.5);
    \draw (1,0) arc (270:180:0.5);
    \draw (1,1) arc (0:-90:0.5);
    \draw (0,1) arc (90:0:0.5);
  }
}

\newcommand{\drawvortex}[1]{%
  \begin{scope}[shift={#1}]
    \draw[fill=gray] (0,0) rectangle (1,1);
    \fill[black!15] (0.1,0.1) rectangle (0.9,0.9);
    \draw (0,0) arc (180:90:0.5);
    \draw (1,0) arc (270:180:0.5);
    \draw (1,1) arc (0:-90:0.5);
    \draw (0,1) arc (90:0:0.5);
  \end{scope}
}

\newcommand{\llb}{\llbracket}
\newcommand{\rrb}{\rrbracket}

\newcommand{\una}{\mathrm{una}}

\newcommand{\drawsides}[6]{
\ifthenelse{\equal{#3}{1}}{
  \draw (#1+0.9, #2+0.1) edge[bend right=10] (#1+0.9, #2+0.9);
  \draw (#1+0.9, #2+0.1) edge[bend left=10] (#1+0.9, #2+0.9);
}{}
\ifthenelse{\equal{#4}{1}}{
  \draw (#1+0.1, #2+0.9) edge[bend right=10] (#1+0.9, #2+0.9);
  \draw (#1+0.1, #2+0.9) edge[bend left=10] (#1+0.9, #2+0.9);
}{}
\ifthenelse{\equal{#5}{1}}{
  \draw (#1+0.1, #2+0.1) edge[bend right=10] (#1+0.1, #2+0.9);
  \draw (#1+0.1, #2+0.1) edge[bend left=10] (#1+0.1, #2+0.9);
}{}
\ifthenelse{\equal{#6}{1}}{
  \draw (#1+0.1, #2+0.1) edge[bend right=10] (#1+0.9, #2+0.1);
  \draw (#1+0.1, #2+0.1) edge[bend left=10] (#1+0.9, #2+0.1);
}{}
}

\title{Two block gluing constructions}

\author{
  Ville Salo
  \ \ \ \ \ \ and\ \ \ \ \
  Ilkka T\"orm\"a
  \\
  Department of Mathematics and Statistics \\
  University of Turku, Finland \\
  \{\texttt{vosalo}, \texttt{iatorm}\}\texttt{@utu.fi}
}

\begin{document}
\maketitle

\begin{abstract}
We prove two existence results about block gluing in two-dimensional SFTs. First, a large class of functions between exponential and logarithmic can be realized as block gluing functions of two-dimensional SFTs. Second, there exists an aperiodic linearly block gluing two-dimensional SFT with entropy dimension $1$. These solve two open questions of Gangloff and Sablik.
\end{abstract}

\section{Introduction}

A two-dimensional subshift is \emph{block gluing} if for any two square patterns, we can find a configuration containing copies of them in specified positions, as long as the gap between the patterns is at least $C$, for some constant $C$. It is well-known that block gluing subshifts have quite different qualitative properties than general subshifts: for instance, a nontrivial block gluing subshift has positive entropy.

In \cite{GaSa21}, Gangloff and Sablik introduced a relaxation of the notion of block gluing, by allowing the gap between two $n$-by-$n$ patterns to depend on $n$. The classical block gluing property corresponds to block gluing with a constant function. Gangloff and Sablik show that linearly block gluing subshifts of finite type (SFTs) can have any $\Pi^0_1$ entropy, agreeing with the class of entropies of general SFTs \cite{HoMe10}.

In \cite{PaSc15a} it is shown that not every $\Pi^0_1$ number appears as the entropy of a constant block gluing SFT, and \cite{GaSa21} generalizes this to $o(\log(n))$ block gluing functions. In this regime, periodic points are dense, and entropies are computable (indeed, one can give an explicit bound on the speed of this computability). It is shown in \cite{PaSc15a} that a large class of numbers does show up as entropies as soon as the dimension is at least $3$. The class of entropies of constant block gluing SFTs is not very well understood in any dimension $d \geq 2$, for example, the ``number-theoretic nature'' of the entropy of the golden mean shift (the binary subshift where horizontally and vertically adjacent $1$s are forbidden) is not understood \cite{Pa12}.

In \cite{GaSa21}, while every $\Pi^0_1$ entropy is realized in linearly block gluing subshifts, the method they use does not produce nontrivial zero-entropy examples. The one-point subshift is an example of a block-gluing subshift with zero entropy, but they leave open whether there are any nontrivial examples.

We give a nontrivial example:

\begin{theorem}
There is an aperiodic linearly block gluing $\Z^2$ SFTs whose entropy dimension is $1$.
\end{theorem}

Remark~3 in \cite{GaSa21} raises the question what the set of possible tight gap functions of block gluing SFT is, specifically they leave open the square root function. This question is repeated by Gangloff et al.\ in \cite{ChGaHeOp26}. Specifically, in \cite{GaHeOp26}, it is shown that for the class of hom-shifts (a subclass of SFTs), the possible block gluing functions have a gap between logarithmic and linear: if the optimal block gluing function is not $O(\log(n))$, then it is $\Omega(n)$. In \cite{ChGaHeOp26}, they ask whether general SFTs might have a similar gap.

We show that the class of block gluing distances is very wide:

\begin{theorem}
For any nondecreasing picture language recognizable function $f$, the class of block gluing functions contains a function asymptotically equal to $f$.
\end{theorem}

Here, ``picture language recognizable'' refers to functions $f$ whose graph is the set of shapes of pictures in a recognizable picture language (defined by tiling rules). This class has been studied in several references, and it is well-known that for example $\sqrt{n}$ belongs to this class. Thus, our theorem solves the problems raised in \cite{GaSa21}. We include a detailed and self-contained study of the recognizability notion in Section~\ref{sec:PictureLanguages}. (In fact, we need only that the \emph{epigraph} of $f$ is recognizable, which might be easier to show for some functions.)

Many questions remain open, for example:

\begin{question}
What are the possible (upper/lower/exact) entropy dimensions of linearly block gluing SFTs?
\end{question}

We recall that the characterization  of \cite{HoMe10} of entropies of SFTs was extended to entropy dimensions by Meyerovitch in \cite{Me11}: for $\Z^d$ SFTs, the class is is $[0, d] \cap \Delta_2^0$. Gangloff showed in \cite{Ga22} that for minimal SFTs, we similarly get $[0,d-1] \cap \Delta^0_2$ as the class of entropy dimensions.

One might guess that in the block gluing case, the characterization stays the same. However, our construction automatically gives entropy dimension $\geq 1$, and even in this range, controlling entropy dimension while having good gluing properties (as is also required by our method) seems difficult. Note that before \cite{GaSa21}, it was not even known whether transitive SFTs can realize all $\Pi^0_1$ entropies, and entropy dimension is even harder to control.

\begin{question}
\label{q:Sublinear}
Are there sublinearly block gluing aperiodic SFTs with entropy $0$?
\end{question}

The question remains open if the aperiodicity requirement is replaced by the assumption that the block gluing function is not logarithmic.

\begin{question}
Can arbitrary $\Pi^0_1$ numbers appear as entropies of sublinearly block gluing SFTs?
\end{question}



\section{Definitions}

\subsection{Notation}

We denote $\llb n \rrb = \{0, 1, \ldots, n-1\}$.
Write $f \approx g$ if $f(n)/g(n) \longrightarrow 1$ as $n \rightarrow \infty$.

\subsection{Subshifts}

Let $A$ be a finite alphabet.
The prodiscrete topology on $A^{\Z^2}$ is the topology generated by the cylinder sets $[p] = \{ x \in A^{\Z^2} \mid x|_D = p \}$ for all finite sets $D \subset \Z^2$ and patterns $p \in A^D$.
A $\Z^2$-\emph{subshift} over $A$ is a set $X \subseteq A^{\Z^2}$ that is topologically closed and invariant under the shift action of $\Z^2$, defined by $\sigma^{\vec v}(x)_{\vec n} = x_{\vec n + \vec v}$.
Every subshift is defined by some set $\mathcal{F}$ of forbidden finite patterns:
\[
X = \{ x \in A^{\Z^2} \mid \forall \vec v \in \Z^2, p \in \mathcal{F} : \sigma^{\vec v}(x) \notin [p] \}
\]
If $\mathcal{F}$ can be chosen finite, then $X$ is a \emph{shift of finite type} (SFT).
A finite pattern $p$ is \emph{globally valid} in a subshift $X$ if $X \cap [p] \neq \emptyset$.
Denote by $\mathcal{L}(X)$ the set of finite patterns that are globally valid in $X$, and by $\lang_n(X)$ the set of globally valid $n \times n$ patterns.

A \emph{Wang tile} is a unit square with colored edges.
If $T$ is a finite set of Wang tiles, $X_T \subseteq T^{\Z^2}$ is the SFT where the colors of shared edges of adjacent tiles must match.

\subsection{Gluing properties}

Let $X$ be a subshift.
The \emph{horizontal block gluing function} of $X$ is the function $B_H : \N \to \N \cup \{\infty\}$ where $B_H(n)$ is the smallest integer $k$ such that for any $p, q \in \mathcal{L}_n(X)$ we have $X \cap [p] \cap \sigma^{(-k,0)}[q] \neq \emptyset$, and $B_H(n) = \infty$ if no such $k$ exists.
In other words, any two globally valid $n \times n$ patterns can be glued together into a single globally valid $(2n+k) \times n$ pattern.
The vertical block gluing function $B_V$ is defined analogously, with $(0,-k)$ in place of $(-k,0)$.
The block gluing function of $X$ is $B(n) = \max(B_H(n), B_V(n))$.

We say $X$ is \emph{$f$-block gluing} if its block gluing function $B$ satisfies $B(n) \leq f(n)$ for all $n \geq 1$. We use similar terminology for the horizontal and vertical notions. We say $X$ is \emph{$O(f)$-block gluing} if its block gluing function satisfies $B = O(f)$, and again similarly for the horizontal and vertical notions.

\begin{remark}
When the function $f$ grows slowly, the shape of the blocks is important. For $f$ growing at least linearly, one can interpret $f$-block gluing as a quantified notion of topological mixing.

Specifically, consider the standard Cantor metric where for $x \neq y$ we set $d(x, y) = 2^{-m}$ where $m$ is minimal such that $x_{\vec v} \neq y_{\vec v}$ for $|\vec v| = 0$.

If $f = \Omega(n)$, then $O(f)$-block gluing is precisely equivalent to the fact that if $U, V$ are open sets each containing a ball of radius at least $2^{-m}$, then we have $\sigma_{\vec v}(U) \cap V \neq \emptyset$ whenever $|\vec v| = \Omega(f(m))$.
\end{remark}


It is often convenient to generalize block gluing to allow gluing rectangles of any shape, and we do this without explicit mention. Suppose $B$ is a gluing function for a subshift $X$. If $p, q \in \mathcal{L}(X)$ are rectangular patterns, and their domains are horizontally or vertically separated by at least $n$ steps, then we allow gluing them if $B(m) \geq n$, where $m$ is the maximum of all dimensions of $p$ and $q$. Then we would indeed be able to glue $p$ and $q$ in $X$ as well: simply extend both into $m \times m$ square patterns.

We say $X$ is \emph{weakly $(f,g)$-block transitive} 
if for all $p, q \in \lang_n(X)$ and $\vec v \in \Z^2$ with $\|\vec v\|_\infty \geq g(n)$,
there exists $\vec w \in \Z^2$ such that $\|\vec w - \vec v\|_\infty \leq f(n)$ and 
$X \cap [p] \cap \sigma^{-\vec w}[q] \neq \emptyset$.
In other words, for any vector $\vec v$ of length at least $g(n)$
we can find an ``error term'' $\vec u$ whose length is bounded by $f(n)$ such that the patterns $p$ and $q$ can be glued with displacement exactly $\vec v + \vec u$.
If $f(n)$ and $g(n)$ are $O(n)$, we say $X$ is weakly linearly block transitive, and if $f(n), g(n) \leq Kn$ for large enough $n$, we say $K$ is a \emph{gluing constant}. 

We say (following \cite{GaSa21}) that $X$ is \emph{$f$-net gluing} if for all $p, q \in \lang_n(X)$ there exists $\vec v \in \Z^2$ such that for all $\vec w \in (n+f(n))(\Z^2 \setminus \{(0,0)\}) + \vec v$, we have $X \cap [p] \cap \sigma^{-\vec w}[q] \neq \emptyset$.

\subsection{Entropy}

Let $X \subseteq A^{\Z^d}$ be a subshift and denote $N_k(X) = |\mathcal{L}_k(X)|$.
The \emph{topological entropy} of $X$ is
\[ h(X) = \lim_{k \rightarrow \infty} \frac{\log N_k(X)}{k}. \]
The \emph{upper entropy dimension} of $X$ is
\[ \bar{D}(X) = \limsup_{k \rightarrow \infty} \frac{\log(\log(N_k(X))}{\log k} \]
and the \emph{lower entropy dimension} is
\[ \underline{D}(X) = \liminf_{k \rightarrow \infty} \frac{\log(\log(N_k(X))}{\log k}. \]
When these agree, we call this the \emph{exact entropy dimension} $D(X)$ of $X$. We collectively refer to these three numbers as the \emph{entropy dimensions} of $X$. We note that the bases of the logarithms do not matter in the entropy dimensions, as long as the bottom logarithm has the same base as the outermost top logarithm (in the definition of entropy, the base does matter). Also, if $h(X) > 0$, then $D(X) = 2$.

\section{Shapes of picture languages}
\label{sec:PictureLanguages}

In this section we review and develop the theory of shapes of recognizable picture languages.
If the reader is only interested in the existence of superlogarithmic and sublinear block gluing functions for $\Z^2$-SFTs, they may first peruse the definitions of this section, then convince themselves that there exists a recognizable picture language such that for all $m \geq 1$, it contains a picture of shape $m \times n$ with $n = \sqrt{m} + O(1)$ (and all pictures are approximately of this shape), and immediately skip to Section~\ref{sec:NoGap}.

A \emph{picture} over a finite alphabet $A$ is a function $p : \llb m \rrb \times \llb n \rrb \to A$ where $m,n\geq 1$ (usually one also allows an empty picture, but we do not). Write $\dom(p) = \llb m \rrb \times \llb n \rrb$ for its domain. Note that we use the usual orientation of $\Z^2$ for pictures, so $m$ is the horizontal axis and the second coordinate increases upward (often one uses matrix indexing with pictures, but we mix pictures with symbolic dynamics on $\Z^2$, where this notation would be quite nonstandard). Write $A^{**}$ for the set of all pictures over alphabet $A$.

The \emph{shape} $|p|$ of a picture $p$ with $\dom(p) = \llb m \rrb \times \llb n \rrb$ is the pair $(m,n) \in \Z_+^2$. Given $(m, n) \in \Z_+^2$, we can conversely construct a unique picture with this shape over the symbol $0$, this is written $\una(m,n)$.

We assume a special symbol $\# \notin A$. Then the \emph{padded picture} $\hat p$ corresponding to $p$ is the pattern $\hat p \in (A \cup \{\#\})^{\llb m+2 \rrb \times \llb n+2\rrb}$ such that $\hat p_{\vec v + (1,1)} = p_{\vec v}$ for $\vec v \in \dom(p)$, and $\hat p_{\vec v} = \#$ for all other $\vec v \in \dom(\hat p)$.

A \emph{picture language} is a set of pictures. A \emph{domino language} is a picture language defined by \emph{forbidden patterns} $B \subset (A \cup \{\#\})^{\{0,1\} \times \{0\}} \cup (A \cup \{\#\})^{\{0\} \times \{0,1\}}$, meaning $p \in L$ if and only if the following holds: for all $q \in B$, whenever $\vec v \in \N^d$ is such that $\vec v + \dom(q) \subset \dom(\hat p)$, we have $\hat p_{\vec v + \vec u} \neq q_{\vec u}$ for some $\vec u \in \dom(q)$. The class $\REC$ of \emph{recognizable picture languages} is the closure of the class of domino languages under symbolwise projections. The definition of $\REC$ is from \cite{GiRe92}, though the class appeared earlier in \cite{InNa77}.

The class $\REC_1$ is the set of recognizable picture languages over a single letter. The \emph{shape} of a picture language $L$, denoted $|L|$, is the set of shapes of its pictures. 
Observe that the shapes of languages in $\REC_1$ are the same as those of languages in $\REC$, and the same as the shapes of domino languages. We say a subset of $\Z_+^2$ is recognizable if the corresponding language of pictures $\una(m,n)$ is in $\REC_1$.

If $f : \Z_+ \to \Z_+$ is a function, the corresponding picture language is the set of pictures $\una(n, f(n))$. This class of functions has been studied in \cite[Section~10]{GiRe96} under the name \emph{recognizable functions}. They prove in particular the following:

\begin{theorem}[\cite{GiRe96}]
Every polynomial with nonnegative integer coefficients is recognizable.
\end{theorem}

It is already observed in \cite{GiRe96} that this class is essentially closed under inversion, and that for instance $\sqrt{n}$ is also recognizable. However, the proof is simply that one can rotate the pictures by $90$ degrees, which gives only a partial function. For our application, it is crucial to have control on the shapes of pictures, and to have full projection to both coordinates. If $f : \Z_+ \to \Z_+$ is a function, then its \emph{hypograph} is $\{(m,n) \;|\; n \leq f(m)\}$, and its \emph{epigraph} is $\{(m,n) \;|\; n \geq f(m)\}$.

\begin{example}
\label{ex:IdEpi}
The epigraph of the identity function is recognizable. Consider the domino language $L_{\textrm{idepi}}$ over alphabet $\{0,1,2\}$ whose allowed patterns (i.e.\ the only patterns not forbidden) are
\[\left\{[\#1], [\#2], [02], [21], [10], [2\#], [1\#], [00], [0\#]\right\}\]
horizontally and 
\[ \{ \begin{bmatrix}1\\\#\end{bmatrix}, \begin{bmatrix}0\\\#\end{bmatrix}, \begin{bmatrix}\#\\2\end{bmatrix}, \begin{bmatrix}\#\\1\end{bmatrix}, \begin{bmatrix}0\\0\end{bmatrix}, \begin{bmatrix}1\\0\end{bmatrix}, \begin{bmatrix}2\\1\end{bmatrix}, \begin{bmatrix}2\\2\end{bmatrix}\} \]
vertically. 

Starting the deduction from the bottom left corner, and proceeding row by row, from left to right on each row, one sees that the tiling of an $m \times n$ box is unique, and the top row is valid only if $n \geq m$. A valid picture is illustrated in Figure~\ref{fig:idepi}.

\begin{figure}
\begin{center}
\begin{tikzpicture}[scale=0.6]
\draw[gray] (0,0) grid (7,8);
\draw[thick] (1,1) rectangle (6,7);

\node () at (0.5, 0.5) {$\#$};
\node () at (1.5, 0.5) {$\#$};
\node () at (2.5, 0.5) {$\#$};
\node () at (3.5, 0.5) {$\#$};
\node () at (4.5, 0.5) {$\#$};
\node () at (5.5, 0.5) {$\#$};
\node () at (6.5, 0.5) {$\#$};
\node () at (0.5, 1.5) {$\#$};
\node () at (1.5, 1.5) {$1$};
\node () at (2.5, 1.5) {$0$};
\node () at (3.5, 1.5) {$0$};
\node () at (4.5, 1.5) {$0$};
\node () at (5.5, 1.5) {$0$};
\node () at (6.5, 1.5) {$\#$};
\node () at (0.5, 2.5) {$\#$};
\node () at (1.5, 2.5) {$2$};
\node () at (2.5, 2.5) {$1$};
\node () at (3.5, 2.5) {$0$};
\node () at (4.5, 2.5) {$0$};
\node () at (5.5, 2.5) {$0$};
\node () at (6.5, 2.5) {$\#$};
\node () at (0.5, 3.5) {$\#$};
\node () at (1.5, 3.5) {$2$};
\node () at (2.5, 3.5) {$2$};
\node () at (3.5, 3.5) {$1$};
\node () at (4.5, 3.5) {$0$};
\node () at (5.5, 3.5) {$0$};
\node () at (6.5, 3.5) {$\#$};
\node () at (0.5, 4.5) {$\#$};
\node () at (1.5, 4.5) {$2$};
\node () at (2.5, 4.5) {$2$};
\node () at (3.5, 4.5) {$2$};
\node () at (4.5, 4.5) {$1$};
\node () at (5.5, 4.5) {$0$};
\node () at (6.5, 4.5) {$\#$};
\node () at (0.5, 5.5) {$\#$};
\node () at (1.5, 5.5) {$2$};
\node () at (2.5, 5.5) {$2$};
\node () at (3.5, 5.5) {$2$};
\node () at (4.5, 5.5) {$2$};
\node () at (5.5, 5.5) {$1$};
\node () at (6.5, 5.5) {$\#$};
\node () at (0.5, 6.5) {$\#$};
\node () at (1.5, 6.5) {$2$};
\node () at (2.5, 6.5) {$2$};
\node () at (3.5, 6.5) {$2$};
\node () at (4.5, 6.5) {$2$};
\node () at (5.5, 6.5) {$2$};
\node () at (6.5, 6.5) {$\#$};
\node () at (0.5, 7.5) {$\#$};
\node () at (1.5, 7.5) {$\#$};
\node () at (2.5, 7.5) {$\#$};
\node () at (3.5, 7.5) {$\#$};
\node () at (4.5, 7.5) {$\#$};
\node () at (5.5, 7.5) {$\#$};
\node () at (6.5, 7.5) {$\#$};

\end{tikzpicture}
\end{center}
\caption{A picture in $L_{\textrm{idepi}}$ of shape $(5, 6)$.}
\label{fig:idepi}
\end{figure}
\end{example}

For us, the epigraph of a function is what allows its implementation as a block gluing function. In this section, we prove lemmas that provide ways to obtain functions with recognizable epigraphs, from known results on recognizable functions. First, epigraphs are no more difficult to implement than graphs:

\begin{lemma}
\label{lem:RECtoEPI}
Let $f$ be a recognizable function. Then the epigraph of $f$ is recognizable.
\end{lemma}

\begin{proof}
Let the graph of $f$ be given as the shape of a domino language $L$ over an alphabet $A$. Consider a new alphabet $A \sqcup \{0\}$ where $0 \notin A$, where the local rules force that $0$-symbols have $0$s above and to the left and right. Then the pictures have $0$s on some set (possibly empty) of top virtual rows, and symbols from $A$ on lower virtual rows. Further require that the bottom row has symbols from $A$ (by forbidding $0$ from having $\#$ below it). Impose the domino rules of $L$ on $A$-symbols, treating $0$s as if they were $\#$s, when they occur in dominoes with $A$-symbols. It is easy to see that we obtain precisely the epigraph of $f$ as the set of shapes.
\end{proof}

Often, it is easier to construct the graph of a function $f$ than to construct its inverse, and Lemma~\ref{lem:Deterministic} and Lemma~\ref{lem:kapprox} deal with the problem of inverting a function. The idea is as in \cite{GiRe96} to rotate the language by $90$ degrees. This works somewhat more smoothly when working with hypographs and epigraphs. Specifically, for natural functions like $n^2$, the $90$ degree rotation of the picture language of the hypograph corresponds to the epigraph of the inverse function.

Unfortunately, the hypograph of a recognizable function is not in general recognizable, under a standard complexity-theoretic conjecture (see Example~\ref{ex:HypoNOT}). We prove however that it is recognizable when either the implementation as a picture language is nicely behaved, or the function itself is sufficiently nicely behaved. Both of these apply to the standard implementation of polynomials with positive integer valued coefficients through deterministic counters.

First, we consider nicely-behaved picture language implementations. We say a domino language $L$ is \emph{upward deterministic} if whenever $p, q \in A^{**}$ are of the same width, agree on the bottom $k \geq 0$ rows, and no domino rules are broken on the $k+2$ bottom rows of the padded pictures $\hat p$ or $\hat q$, then $p$ and $q$ also agree on the $(k+1)$th row (starting indexing of rows from $1$). This implies that for any two pictures in $L$ of equal width, one is a northward extension of the other.
The following lemma is essentially weaker than Lemma~\ref{lem:gdivf}, but is easier to prove.

\begin{lemma}
\label{lem:Deterministic}
Let $f$ be a recognizable function, and suppose its graph is the shape of a domino language which is upward deterministic. Then the hypograph of $f$ is recognizable.
\end{lemma}

\begin{proof}
Let $L$ be the implementation of $f$ with an upward deterministic domino language, with alphabet $A$. Construct a new domino language $L'$ over alphabet $(A \times \{0,1\}^2) \sqcup A$. Each cell has an $A$-symbol, and some cells have two extra bits. We enforce the domino rules of $L$ on the $A$-symbols.

We use additional domino rules to ensure the following:
\begin{itemize}
\item The top row has $A$-symbols without extra bits, and all other rows have the extra bits (this is only a function of whether the top neighbor is $\#$).
\item The second bit of a cell $\vec v$ is always equal to the first bit of the right neighbor $\vec v + (1,0)$ unless either is $\#$.
\item The first bit is $0$ on the leftmost cell of each row, and the second bit is $1$ on the rightmost cell of each row.
\item The first and second bit of each cell are one of $(0,0), (0,1), (1,1)$, and if a cell has bits $(0,1)$, the top neighbor $\#$ would break the local domino rules of $L$.
\end{itemize}
This has the effect that each row except the topmost one has a sequence of bits ``on the edges'' between cells (thinking of the first and second bit as being on the edges of the cell).
Each row has a unique cell where the bit changes from $0$ to $1$, and this cell does not admit $\#$ as a top neighbor in $L$.

We claim that this implements the hypograph. To see this, first suppose $(m, n)$ is in the hypograph, so that $n \leq f(m)$. Then there is a picture $p \in L$ with $|p| = (m, f(m))$. Take the bottom $n$ rows of this picture. Since $L$ implements the graph of $f$, certainly the bottommost $i$ rows do not form a valid picture for any $i < n$. Since they are the bottommost $i$ rows of the valid picture $p$ and $i < f(m)$, there must be a cell on the $i$th row that does not admit $\#$ as its top neighbor. This allows us to pick the sequence of bits on this row, so as to get a valid picture in $L'$. (Note that in $(m, f(m))$, each cell of the top row can be matched with $\#$, but the sequence of bits is not required on this row.)

Next, consider $(m, n) \in |L'|$, say $|q| = (m, n)$ for $q \in L'$. Since $|L|$ is the graph of $f$, there is a picture $p \in L$ of shape $|p| = (m, f(m))$. By upward determinism, the bottommost $i$ rows of $p$ and $q$ are equal when we ignore the extra bits, as long as $i \leq \min(f(m), n)$. In particular, if $n \geq f(m)$, then the $f(m)$th row of $q$ from the bottom is precisely the top row of $p$. Since $p \in L$, each cell of this row must allow $\#$ as its top neighbor. So we must have $n \leq f(m)$ in this case, as otherwise the extra bits on $q$ could not possibly exist. We conclude that $(m, n)$ is in the hypograph of $f$.
\end{proof}

Next, we consider nicely-behaved functions. The \emph{$k$-approximate hypograph} of $f$ is $\{(m,n) \;|\; n \leq f(m) \leq kn\}$.

\begin{lemma}
\label{lem:kapprox}
Let $f : \Z_+ \to \Z_+$ be a recognizable function. Then its $k$-approximate hypograph is recognizable.
\end{lemma}

\begin{proof}
We use $k$ layers of virtual cells in each actual cell, and on every second layer, we vertically invert the virtual cell so that its bottom is actually on the top. Gluing the $(2i+1)$th layer to the $(2i+2)$th layer on the top (starting indexing of layers from $1$), and the $(2i+2)$th to the $(2i+3)$th on the bottom, and then the the last layers to the boundary, we can simulate a height-$kn$ picture when the actual height is $n$. Furthermore, each virtual cell knows which $k$th of $kn$ it belongs to (i.e.\ which layer it belongs to).

On top of this construction, we can implement the $k$-approximate hypograph as follows. We let $L$ be a domino language over alphabet $A$, whose shape is $\{(m, f(m)) \;|\; m \in \Z_+\}$. We consider on the virtual cells a picture language over alphabet $A \sqcup \{0\}$ where $0 \notin A$, where the local rules force that $0$-symbols have $0$s above and to the left and right (everything in the relative sense of the virtual cells). Then the pictures have $0$s on some set (possibly empty) of top virtual rows, and symbols from $A$ on lower virtual rows. We now additionally impose the domino rules of $L$, treating $0$ as if they were $\#$, when it occurs in dominoes with $A$-symbols. Finally, we require that the first layer contains only $A$-symbols to check $n \leq f(m)$.
\end{proof}

\begin{lemma}
\label{lem:kthing}
Let $f : \Z_+ \to \Z_+$ be an eventually nondecreasing recognizable function. Suppose that for some $k' \in \Z_+$, we have $f(m+1) \leq k'f(m) + 1$ for all large enough $m$. Then the hypograph of $f$ is recognizable.
\end{lemma}

\begin{proof}
We may assume that $f$ is nondecreasing,
since recognizable sets are easily seen to be closed under finite modifications.
We may also assume $f(m+1) \leq k'f(m) + 1$ is true for all $m \in \Z_+$, by simply increasing the value of $k'$.

Let now $H$ be the hypograph of $f$, and let $H'$ be the $k$-approximate hypograph of $f$, where $k = \max(f(1), k')$. By Lemma~\ref{lem:kapprox}, $H'$ is recognizable by some domino language $L$ on alphabet $A$.

We apply a similar construction as in Lemma~\ref{lem:kapprox} using a symbol $0 \notin A$, but rotated $90$ degrees. Specifically, on the alphabet $A \sqcup \{0\}$, require that all cells to the right, above and below $0$s are $0$, the bottom left corner of the picture contains an $A$-symbol, and treat $0$-neighbors of $A$-symbols as containing $\#$ for the domino rules of $L$.
Hence, the set
\[
H'' = \{(m, n) \;|\; \exists \ell \leq m: (\ell, n) \in H'\} = \{(m, n) \;|\; \exists \ell \leq m: n \leq f(\ell) \leq kn\}
\]
we obtain in this way is recognizable.

We claim that $H''$ is equal to the hypograph $H$ of $f$.
First, if $(m, n) \in H''$, then for some $\ell \leq m$ we have $n \leq f(\ell) \leq kn$. In particular, $n \leq f(\ell) \leq f(m)$ since $f$ is nondecreasing, so $(m,n) \in H$.

Conversely, let $(m,n) \in H$, i.e.\ $n \leq f(m)$.
If $f(m) \leq kn$, then $(m,n) \in H' \subset H''$ and we are done.
Suppose then that $f(m) > kn$ and consider the largest $1 \leq \ell < m$ with $f(\ell) \leq kn$; since $f(1) \leq k \leq kn$, such an $\ell$ exists.
By the assumption on $f$, we have
\[
k' f(\ell) + 1 \geq f(\ell+1) \geq kn + 1 \geq k'n+1,
\]
which implies $f(\ell) \geq n$.
Hence $(\ell,n) \in H'$ and $(m,n) \in H''$.
\end{proof}

\begin{example}
We illustrate the above lemmas for $f(n) = n^2$. In this case $f(n+1) = n^2 + 2n + 1 \leq 3n^2 + 1$ for all $n$, so we can take $k' = 3$. Since $f(1) = 1$, we can take $k = 3$. (Note also that we do not need the preprocessing on $f$ discussed in the proof.) First, we have by assumption the recognizability of the graph of $f$. In Lemma~\ref{lem:kapprox}, we construct the $3$-approximate hypograph $H'$ of $f$, i.e.\ the set $\{(m,n) \;|\; n \leq m^2 \leq 3n\}$. We then extend the set to $H''$, by including every $(m+\ell, n)$ such that $(m,n)$ is already in the graph. These sets are illustrated in Figure~\ref{fig:fsteps}.
\begin{figure}
\begin{center}
\begin{tikzpicture}[scale=0.5]
\draw[fill] (0,0) rectangle (1,1);
\draw[fill] (1,3) rectangle (2,4);
\draw[fill] (2,8) rectangle (3,9);
\draw[fill] (3,15) rectangle (4,16);

\draw[fill,black!60!white] (1,1) rectangle (2,3);
\draw[fill,black!60!white] (2,2) rectangle (3,8);
\draw[fill,black!60!white] (3,5) rectangle (4,15);
\draw[fill,black!60!white] (4,8) rectangle (5,18);
\draw[fill,black!60!white] (5,11) rectangle (6,18);
\draw[fill,black!60!white] (6,16) rectangle (7,18);

\draw[fill,black!30!white] (1,0) rectangle (8,1);
\draw[fill,black!30!white] (2,1) rectangle (8,2);
\draw[fill,black!30!white] (3,2) rectangle (8,5);
\draw[fill,black!30!white] (4,5) rectangle (8,8);
\draw[fill,black!30!white] (5,8) rectangle (8,11);
\draw[fill,black!30!white] (6,11) rectangle (8,16);
\draw[fill,black!30!white] (7,16) rectangle (8,18);
\draw (0,0) grid (8,18);
\end{tikzpicture}
\end{center}
\caption{In black, we show the two-dimensional $10$-by-$16$ prefix of the graph of $n^2$. In dark gray, we show $H' \setminus H$. In light gray, we show $H'' \setminus H'$. The hypograph of $n^2$ is precisely the shaded area $H''$.}
\label{fig:fsteps}
\end{figure}
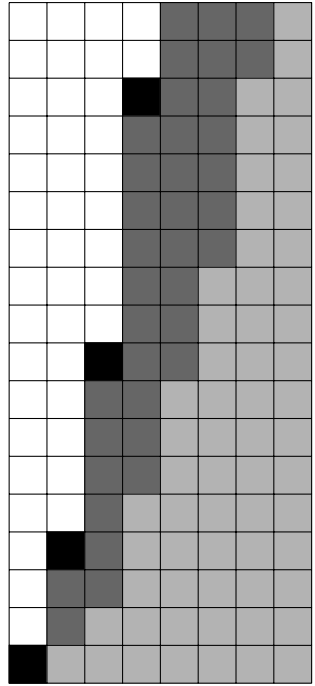

This illustrates the simple geometric idea: the submultiplicativity assumption on $f$ that $f(m+1) \leq k'f(m) + 1$ guarantees that in the $k$-approximate hypograph, the new pairs in $H'$ descend far enough down from the graph of $f$ so that the cells below them, which are also supposed to be in the hypograph, can also be obtained by extending lower values of the graph to the right. (And since $f$ is nondecreasing, doing this does not add any pairs that do not belong to the hypograph.)
\end{example}

The following formalizes the idea that the rotation by $90$ degrees turns the hypograph of a function into the epigraph of the inverse function.

\begin{lemma}
\label{lem:InverseEpigraph}
Let $f : \R_+ \to \R_+$ be an increasing function whose restriction to positive integers takes values in positive integers. Suppose that the hypograph of $f|_{\Z_+} : \Z_+ \to \Z_+$ is recognizable. Then the epigraph of $\lceil f^{-1} \rceil|_{\Z_+} : \Z_+ \to \Z_+$ (the ceiling of the inverse function) is recognizable.
\end{lemma}

\begin{proof}
The transpose of the hypograph of $f$ is recognizable, since recognizable picture language are closed under $90$ degree rotation. This is the set
\begin{align*}
\{(m,n) \;|\; m \leq f(n)\} &= \{(m,n) \;|\; n \geq f^{-1}(m) \} \\
&= \{(m,n) \;|\; n \geq \lceil f^{-1}(m) \rceil \}
\end{align*}
where the last equality is because $n$ is an integer. This is by definition the restriction of the epigraph of $\lceil f^{-1} \rceil$.
\end{proof}

\begin{example}
\label{ex:HypoNOT}
We show that if the hypograph of every recognizable function $f : \Z_+ \to \Z_+$ is recognizable, then $\text{NE} = \text{co-NE}$. These denote the nondeterministic and co-nondeterministic exponential time classes $\text{NTIME}(O(k^n))$ and $\text{co-NTIME}(O(k^n))$ respectively. It is well-known \cite{DuJoMaMo12} that $\text{NE} = \text{co-NE}$ is equivalent to $\text{NP} \cap \text{UN} = \text{co-NP} \cap \text{UN}$ by a padding argument, where UN is the class of all languages over alphabet $\{1\}$. (The correspondence is $1^n \in \{1\}^* \leftrightarrow w \in \{0,1\}^*$ where $w$ is the binary representation of $n$.)

Let $L \in \text{NP} \cap \text{UN}$, and suppose $1^n \in L$ can be verified in time $n^d$ by a nondeterministic Turing machine $T$ for an integer constant $d \geq 1$. Then we build a recognizable function $f$ fitting the following description: If $m = k^d$ for some $k$, and $1^k$ is accepted by the machine $T$ in time $k^d$, then $f(m) = m$. If $m$ is not of the form $k^d$, then also $f(m) = m$. If $m = k^d$ for some $d$, and $T$ does not accept $1^k$ in time $k^d$, then $f(m) > m$.

We explain below how to achieve this. First, let us explain how the hypograph of this function being recognizable implies the complexity-theoretic claim. Indeed suppose $K$ is a picture language whose shape is the hypograph of $f$. We observe that $1^k \notin L$ if and only if $(k^d, k^d+1)$ is in the hypograph of $f$. But this can be verified by exhibiting a picture of this shape in $K$, which in turn can be done in time $O(k^{2d})$ by a nondeterministic Turing machine. We conclude that $L^c$ is in co-NP.

We only explain informally how to construct the picture language. Observe that we can prove that the width is of the form $m = k^d$ by running $d$ unary counters from left to right on the $k$ bottommost rows. We can then send diagonal signals to have $k$ marked on the leftmost cells of the bottom row. Now we can simply carry out a successful computation of the Turing machine on this input, and ensure square shapes by using a diagonal signal, to realize the $n = f(k^d) = k^d$ part of the graph.

We can prove that $m$ is not of the form $k^d$ by guessing a number $k$ such that $k^d < m < (k+1)^d$, and running two counters (on independent tracks of the alphabet), so that one has time to finish strictly before reaching the right side of the picture, and the other does not finish. On these inputs, we can again realize square shapes by a diagonal signal.

Finally, on inputs of the form $m = k^d$ where the Turing machine does not finish in time, we prove that no computation of length $k^d$ succeeds. For this, we go through all possible words of length $k^d$ over $Q \times \Sigma$, and for each of them, we check (by a classical deterministic computation) whether they are the valid trace (the sequence of the tape symbols and states of the head during the computation) of an accepting computation of $T$ on input $1^k$. Note that to check a trace is valid, it suffices to check that whenever a tape symbol is first encountered, its symbol is consistent with input $1^k$, that when the machine revisits a cell it sees the symbol it wrote there last time, and that the sequence of states evolves according to the (nondeterministic) transition function of $T$.

After we have checked that no word is the trace of an accepting computation, we accept the input $k^d$ (by allowing a transition to $\#$s on the top of the picture), giving a unique picture of shape $(k^d, f(k^d)) = (k^d, \sum_{w \in (Q \times \Sigma)^{m^d}} t(w))$ where $t(w)$ is the time our deterministic computation takes to verify that the word $w$ does not represent a valid computation. (Of course, we can alternatively have $f$ take a larger but nicer value, by running a counter on another track.) \qee
\end{example}

We now show another method based for constructing epigraphs based on determinism, which naturally gives the epigraph of a function of the form $\Theta(n^r)$ for $r$ a positive rational number.

\begin{lemma}
\label{lem:gdivf}
Let $f : \R_+ \to \R_+$ be an increasing function with $f(n) \geq n$ for all $n \in \Z_+$, whose restriction to $\Z_+$ takes values in $\Z_+$ and is recognizable with upward deterministic picture language. Let $g : \Z_+ \to \Z_+$ be a nondecreasing function with $g(n) \geq n$ for all $n$, which is recognizable. Then $m \mapsto g(\lfloor f^{-1}(m) \rfloor)$ has recognizable epigraph.
\end{lemma}

\begin{proof}
We define a domino language $L$ as follows.
In a picture of shape $(m, n)$, we guess a number $k$ by marking the $k$ bottommost cells on the leftmost column and the $k$ leftmost cells on the bottommost row, synchronized by a diagonal signal. We implement two independent copies of the domino language of $f$ rotated by $90$ degrees, so the computation is deterministic to the right, and we check that $f(k) \leq m$, $f(k+1) > m$ by checking that the first computation has time to finish (possibly on the rightmost row), and the second does not. Next, we mark the $k$ leftmost columns with a special color, and in this $k \times n$ region, check a picture from the language realizing the epigraph of $g$ (which by Lemma~\ref{lem:RECtoEPI} is recognizable). 

We claim that $|L|$ is the epigraph of $m \mapsto g(\lfloor f^{-1}(m) \rfloor)$. Namely, suppose $m \in \Z_+$, $k = \lfloor f^{-1}(m) \rfloor$ and $n = g(k)$. We need to show $(m, \ell) \in |L|$ for any $\ell \geq n$. Observe first that $k \leq \min(m, n)$, so we can indeed mark the $k$ cells on the leftmost column and bottommost row. Since $k = \lfloor f^{-1}(n) \rfloor$, we have $k \leq f^{-1}(n) < k+1$, so $f(k) \leq n < f(k+1)$. Then the unique deterministic computation to the right succeeds for $k$ and fails for $k+1$, as required. Since $n = g(k)$, the pair $(k, \ell)$ is on the epigraph of $g$, and thus we can find a picture to put on the $k$ leftmost columns.

Conversely, consider any picture $p \in L$ of shape $(m, n)$. By the determinism of the domino language of $f$, the only possible choice of $k$ is $k = \lfloor f^{-1}(m) \rfloor$. Since we check a picture from the language realizing the epigraph of $g$ on the leftmost $k$ columns, $(k, n)$ must be on the epigraph of $g$, meaning
\[ g(\lfloor f^{-1}(m) \rfloor) = g(k) \leq n. \qedhere \]
\end{proof}

\begin{example}
\label{ex:23}
We show with semi-abstract pictures what the construction looks like for $f(k) = k^3, g(k) = k^2$, when $k = 3$ and $(m, n) = (30, 11)$. Note that $\lfloor f(m) \rfloor = \lfloor \sqrt[3]{30}\rfloor = 3$ and $11 \geq 9 = g(3)$, so this should indeed be on the epigraph. We mark $k$ cells on the left boundary and the bottom boundary, and synchronize their sizes by diagonal signals.

We mark $k$ cells on the left boundary, and check $f(k) \leq m < f(k+1)$ with black, gray and white dots. The calculation for $k = 3$ has time to end, but the one for $k = 4$ does not. For readability, we have put the $k = 4$ calculation on the rows 4--7 instead of on another track on the bottom $4$ rows (this could also be done in the proof, but it requires $g(k) \geq 2k+1$). This is illustrated in Figure~\ref{fig:fpart}. Note that the movement of the dots requires long-range signals, which are hidden from the figure, but their implementation is standard.

\begin{figure}
\begin{center}
\begin{tikzpicture}[scale=0.33]
\draw[line width=3,black!30!white] (0,0) rectangle (30,3);
\draw[line width=3,black!30!white] (0,3) rectangle (30,7);

\draw[fill=black] (0.2,0.4) circle (0.2);
\draw[fill=gray] (0.5,0.5) circle (0.2);
\draw[fill=white] (0.7,0.6) circle (0.2);
\draw[fill=black] (1.2,0.4) circle (0.2);
\draw[fill=gray] (1.5,0.5) circle (0.2);
\draw[fill=white] (1.7,1.6) circle (0.2);
\draw[fill=black] (2.2,0.4) circle (0.2);
\draw[fill=gray] (2.5,0.5) circle (0.2);
\draw[fill=white] (2.7,2.6) circle (0.2);
\draw[fill=black] (3.2,0.4) circle (0.2);
\draw[fill=gray] (3.5,1.5) circle (0.2);
\draw[fill=white] (3.7,0.6) circle (0.2);
\draw[fill=black] (4.2,0.4) circle (0.2);
\draw[fill=gray] (4.5,1.5) circle (0.2);
\draw[fill=white] (4.7,1.6) circle (0.2);
\draw[fill=black] (5.2,0.4) circle (0.2);
\draw[fill=gray] (5.5,1.5) circle (0.2);
\draw[fill=white] (5.7,2.6) circle (0.2);
\draw[fill=black] (6.2,0.4) circle (0.2);
\draw[fill=gray] (6.5,2.5) circle (0.2);
\draw[fill=white] (6.7,0.6) circle (0.2);
\draw[fill=black] (7.2,0.4) circle (0.2);
\draw[fill=gray] (7.5,2.5) circle (0.2);
\draw[fill=white] (7.7,1.6) circle (0.2);
\draw[fill=black] (8.2,0.4) circle (0.2);
\draw[fill=gray] (8.5,2.5) circle (0.2);
\draw[fill=white] (8.7,2.6) circle (0.2);
\draw[fill=black] (9.2,1.4) circle (0.2);
\draw[fill=gray] (9.5,0.5) circle (0.2);
\draw[fill=white] (9.7,0.6) circle (0.2);
\draw[fill=black] (10.2,1.4) circle (0.2);
\draw[fill=gray] (10.5,0.5) circle (0.2);
\draw[fill=white] (10.7,1.6) circle (0.2);
\draw[fill=black] (11.2,1.4) circle (0.2);
\draw[fill=gray] (11.5,0.5) circle (0.2);
\draw[fill=white] (11.7,2.6) circle (0.2);
\draw[fill=black] (12.2,1.4) circle (0.2);
\draw[fill=gray] (12.5,1.5) circle (0.2);
\draw[fill=white] (12.7,0.6) circle (0.2);
\draw[fill=black] (13.2,1.4) circle (0.2);
\draw[fill=gray] (13.5,1.5) circle (0.2);
\draw[fill=white] (13.7,1.6) circle (0.2);
\draw[fill=black] (14.2,1.4) circle (0.2);
\draw[fill=gray] (14.5,1.5) circle (0.2);
\draw[fill=white] (14.7,2.6) circle (0.2);
\draw[fill=black] (15.2,1.4) circle (0.2);
\draw[fill=gray] (15.5,2.5) circle (0.2);
\draw[fill=white] (15.7,0.6) circle (0.2);
\draw[fill=black] (16.2,1.4) circle (0.2);
\draw[fill=gray] (16.5,2.5) circle (0.2);
\draw[fill=white] (16.7,1.6) circle (0.2);
\draw[fill=black] (17.2,1.4) circle (0.2);
\draw[fill=gray] (17.5,2.5) circle (0.2);
\draw[fill=white] (17.7,2.6) circle (0.2);
\draw[fill=black] (18.2,2.4) circle (0.2);
\draw[fill=gray] (18.5,0.5) circle (0.2);
\draw[fill=white] (18.7,0.6) circle (0.2);
\draw[fill=black] (19.2,2.4) circle (0.2);
\draw[fill=gray] (19.5,0.5) circle (0.2);
\draw[fill=white] (19.7,1.6) circle (0.2);
\draw[fill=black] (20.2,2.4) circle (0.2);
\draw[fill=gray] (20.5,0.5) circle (0.2);
\draw[fill=white] (20.7,2.6) circle (0.2);
\draw[fill=black] (21.2,2.4) circle (0.2);
\draw[fill=gray] (21.5,1.5) circle (0.2);
\draw[fill=white] (21.7,0.6) circle (0.2);
\draw[fill=black] (22.2,2.4) circle (0.2);
\draw[fill=gray] (22.5,1.5) circle (0.2);
\draw[fill=white] (22.7,1.6) circle (0.2);
\draw[fill=black] (23.2,2.4) circle (0.2);
\draw[fill=gray] (23.5,1.5) circle (0.2);
\draw[fill=white] (23.7,2.6) circle (0.2);
\draw[fill=black] (24.2,2.4) circle (0.2);
\draw[fill=gray] (24.5,2.5) circle (0.2);
\draw[fill=white] (24.7,0.6) circle (0.2);
\draw[fill=black] (25.2,2.4) circle (0.2);
\draw[fill=gray] (25.5,2.5) circle (0.2);
\draw[fill=white] (25.7,1.6) circle (0.2);
\draw[fill=black] (26.2,2.4) circle (0.2);
\draw[fill=gray] (26.5,2.5) circle (0.2);
\draw[fill=white] (26.7,2.6) circle (0.2);

\draw[fill=black] (0.2,3.4) circle (0.2);
\draw[fill=gray] (0.5,3.5) circle (0.2);
\draw[fill=white] (0.7,3.6) circle (0.2);
\draw[fill=black] (1.2,3.4) circle (0.2);
\draw[fill=gray] (1.5,3.5) circle (0.2);
\draw[fill=white] (1.7,4.6) circle (0.2);
\draw[fill=black] (2.2,3.4) circle (0.2);
\draw[fill=gray] (2.5,3.5) circle (0.2);
\draw[fill=white] (2.7,5.6) circle (0.2);
\draw[fill=black] (3.2,3.4) circle (0.2);
\draw[fill=gray] (3.5,3.5) circle (0.2);
\draw[fill=white] (3.7,6.6) circle (0.2);
\draw[fill=black] (4.2,3.4) circle (0.2);
\draw[fill=gray] (4.5,4.5) circle (0.2);
\draw[fill=white] (4.7,3.6) circle (0.2);
\draw[fill=black] (5.2,3.4) circle (0.2);
\draw[fill=gray] (5.5,4.5) circle (0.2);
\draw[fill=white] (5.7,4.6) circle (0.2);
\draw[fill=black] (6.2,3.4) circle (0.2);
\draw[fill=gray] (6.5,4.5) circle (0.2);
\draw[fill=white] (6.7,5.6) circle (0.2);
\draw[fill=black] (7.2,3.4) circle (0.2);
\draw[fill=gray] (7.5,4.5) circle (0.2);
\draw[fill=white] (7.7,6.6) circle (0.2);
\draw[fill=black] (8.2,3.4) circle (0.2);
\draw[fill=gray] (8.5,5.5) circle (0.2);
\draw[fill=white] (8.7,3.6) circle (0.2);
\draw[fill=black] (9.2,3.4) circle (0.2);
\draw[fill=gray] (9.5,5.5) circle (0.2);
\draw[fill=white] (9.7,4.6) circle (0.2);
\draw[fill=black] (10.2,3.4) circle (0.2);
\draw[fill=gray] (10.5,5.5) circle (0.2);
\draw[fill=white] (10.7,5.6) circle (0.2);
\draw[fill=black] (11.2,3.4) circle (0.2);
\draw[fill=gray] (11.5,5.5) circle (0.2);
\draw[fill=white] (11.7,6.6) circle (0.2);
\draw[fill=black] (12.2,3.4) circle (0.2);
\draw[fill=gray] (12.5,6.5) circle (0.2);
\draw[fill=white] (12.7,3.6) circle (0.2);
\draw[fill=black] (13.2,3.4) circle (0.2);
\draw[fill=gray] (13.5,6.5) circle (0.2);
\draw[fill=white] (13.7,4.6) circle (0.2);
\draw[fill=black] (14.2,3.4) circle (0.2);
\draw[fill=gray] (14.5,6.5) circle (0.2);
\draw[fill=white] (14.7,5.6) circle (0.2);
\draw[fill=black] (15.2,3.4) circle (0.2);
\draw[fill=gray] (15.5,6.5) circle (0.2);
\draw[fill=white] (15.7,6.6) circle (0.2);
\draw[fill=black] (16.2,4.4) circle (0.2);
\draw[fill=gray] (16.5,3.5) circle (0.2);
\draw[fill=white] (16.7,3.6) circle (0.2);
\draw[fill=black] (17.2,4.4) circle (0.2);
\draw[fill=gray] (17.5,3.5) circle (0.2);
\draw[fill=white] (17.7,4.6) circle (0.2);
\draw[fill=black] (18.2,4.4) circle (0.2);
\draw[fill=gray] (18.5,3.5) circle (0.2);
\draw[fill=white] (18.7,5.6) circle (0.2);
\draw[fill=black] (19.2,4.4) circle (0.2);
\draw[fill=gray] (19.5,3.5) circle (0.2);
\draw[fill=white] (19.7,6.6) circle (0.2);
\draw[fill=black] (20.2,4.4) circle (0.2);
\draw[fill=gray] (20.5,4.5) circle (0.2);
\draw[fill=white] (20.7,3.6) circle (0.2);
\draw[fill=black] (21.2,4.4) circle (0.2);
\draw[fill=gray] (21.5,4.5) circle (0.2);
\draw[fill=white] (21.7,4.6) circle (0.2);
\draw[fill=black] (22.2,4.4) circle (0.2);
\draw[fill=gray] (22.5,4.5) circle (0.2);
\draw[fill=white] (22.7,5.6) circle (0.2);
\draw[fill=black] (23.2,4.4) circle (0.2);
\draw[fill=gray] (23.5,4.5) circle (0.2);
\draw[fill=white] (23.7,6.6) circle (0.2);
\draw[fill=black] (24.2,4.4) circle (0.2);
\draw[fill=gray] (24.5,5.5) circle (0.2);
\draw[fill=white] (24.7,3.6) circle (0.2);
\draw[fill=black] (25.2,4.4) circle (0.2);
\draw[fill=gray] (25.5,5.5) circle (0.2);
\draw[fill=white] (25.7,4.6) circle (0.2);
\draw[fill=black] (26.2,4.4) circle (0.2);
\draw[fill=gray] (26.5,5.5) circle (0.2);
\draw[fill=white] (26.7,5.6) circle (0.2);
\draw[fill=black] (27.2,4.4) circle (0.2);
\draw[fill=gray] (27.5,5.5) circle (0.2);
\draw[fill=white] (27.7,6.6) circle (0.2);
\draw[fill=black] (28.2,4.4) circle (0.2);
\draw[fill=gray] (28.5,6.5) circle (0.2);
\draw[fill=white] (28.7,3.6) circle (0.2);
\draw[fill=black] (29.2,4.4) circle (0.2);
\draw[fill=gray] (29.5,6.5) circle (0.2);
\draw[fill=white] (29.7,4.6) circle (0.2);

\draw (0,0) grid (30, 11);
\end{tikzpicture}
\end{center}
\caption{Illustration of the checks $f(k) \leq m < f(k+1)$ by standard counter constructions.} 
\label{fig:fpart}
\end{figure}
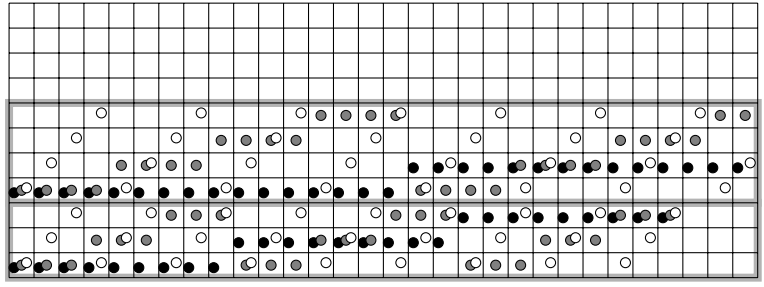

We now use the $k$ leftmost columns to check a picture realizing the epigraph of $g$ using black and white dots. These are illustrated in Figure~\ref{fig:gpart}. Note that the diagonal signals synchronizing vertical and horizontal choices of $k$ are not not visible in the pictures. \qee
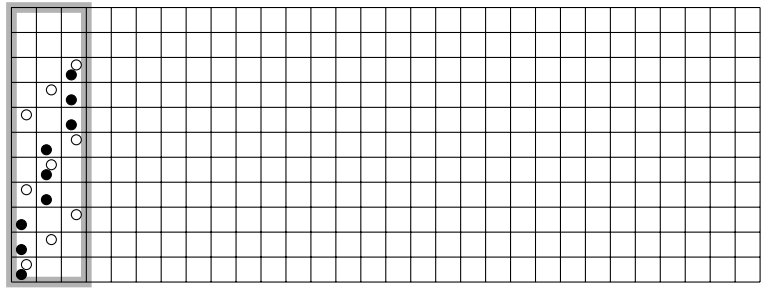
\begin{figure}
\begin{center}
\begin{tikzpicture}[scale=0.33]
\draw[line width=4,black!30!white] (0,0) rectangle (3,11);




\draw[fill=black] (0.4,0.3) circle (0.2);
\draw[fill=white] (0.6,0.7) circle (0.2);
\draw[fill=black] (0.4,1.3) circle (0.2);
\draw[fill=white] (1.6,1.7) circle (0.2);
\draw[fill=black] (0.4,2.3) circle (0.2);
\draw[fill=white] (2.6,2.7) circle (0.2);
\draw[fill=black] (1.4,3.3) circle (0.2);
\draw[fill=white] (0.6,3.7) circle (0.2);
\draw[fill=black] (1.4,4.3) circle (0.2);
\draw[fill=white] (1.6,4.7) circle (0.2);
\draw[fill=black] (1.4,5.3) circle (0.2);
\draw[fill=white] (2.6,5.7) circle (0.2);
\draw[fill=black] (2.4,6.3) circle (0.2);
\draw[fill=white] (0.6,6.7) circle (0.2);
\draw[fill=black] (2.4,7.3) circle (0.2);
\draw[fill=white] (1.6,7.7) circle (0.2);
\draw[fill=black] (2.4,8.3) circle (0.2);
\draw[fill=white] (2.6,8.7) circle (0.2);

\draw (0,0) grid (30, 11);
\end{tikzpicture}
\end{center}
\caption{Illustration of the check $g(k) \leq n$ by a standard counter construction.} 
\label{fig:gpart}
\end{figure}
\end{example}

We are now ready to implement epigraphs of rational powers.

\begin{proposition}
\label{prop:PosRatPow}
Let $r \in \Q$ be a positive rational number. Then there is a function $f(n) \approx n^r$ whose epigraph is recognizable.
\end{proposition}

\begin{proof}
We have $r = a/b$ where $a, b \geq 1$. The function $f(x) = x^b$ is increasing, satisfies $f(n) \geq n$ for all $n \in \Z_+$, takes positive integers to positive integers, and (as in Example~\ref{ex:HypoNOT}) is recognizable with an upward deterministic domino language.
The function $g(n) = n^a$ is also nondecreasing and recognizable, with $g(n) \geq n$ for all $n$.
By Lemma~\ref{lem:gdivf}, $h(n) = \lfloor \sqrt[b]{n} \rfloor^a$ has recognizable epigraph. Clearly $h \approx n^r$. 
\end{proof}

The key point of studying epigraphs is that in our implementation, the following optimization problems pop up naturally.

\begin{definition}
Let $L \subset A^{**}$ be a picture language. Define 
\[ S_H(n) = \inf\{m \;|\; (m,n) \in |L|\} \in \N \cup \{\infty\} \]
as the \emph{horizontal size function}, and
\[ S'_H(n) = \sup\{S_H(n') \;|\; n' \leq n, S_H(n') < \infty \} \]
as the \emph{relaxed horizontal size function}. Symmetrically we define the vertical variants
\[ S_V(n) = \inf\{m \;|\; (n,m) \in |L|\} \in \N \cup \{\infty\} \]
\[ S'_V(n) = \sup\{S_V(n') \;|\; n' \leq n, S_V(n') < \infty \}. \]
Finally define $S(n) = \max(S_H(n), S_V(n))$.
\end{definition}

For an epigraph, these reduce to triviality:

\begin{lemma}
\label{lem:sthing}
Let $f : \Z_+ \to \Z_+$ be a nondecreasing function whose epigraph is recognizable. Then there exists a domino language such that
\[ S'_V(n) = f(n) \]
and 
\[ S'_H(n) = 1 \]
for all $n$.
\end{lemma}

\begin{proof}
The epigraph of $f$ is recognizable. By definition, for the corresponding domino language we have $S_V(n) = f(n)$, and $S_V'(n) = S_V(n)$ since $f$ is nondecreasing. On the other hand, $f(1) \geq 1$ so $(1,n) \in |L|$ for all $n$, and we conclude $S_H(n) = 1$ for all $n \in \Z_+$. Thus also $S_H'(n) = 1$ for all $n$.
\end{proof}

Finally, we recall an interesting theoretical characterization from the literature \cite{BeGoLo09}, which to some readers may clarify the class of recognizable functions and epigraphs, although we did not directly use it here.
Define the \emph{coded shape} of a pattern $p$ with $\dom(p) = \llb m \rrb \times \llb n \rrb$ as a word over alphabet $\{{\circ}, v, h\}$, defined as follows
\[ \cshape(p) = \begin{cases}
{\circ}^n h {\circ}^{m-n-1}, & \mbox{if } m > n, \\
{\circ}^m v {\circ}^{n-m-1}, & \mbox{if } m < n, \\
{\circ}^n, & \mbox{if } m = n.
\end{cases} \]
Write $Q$ for the set of \emph{quasi-unary words}, which is the regular language given by the regular expression ${\circ}^* (\epsilon + (h+v) {\circ}^*)$. This is the set of coded shapes of all possible pictures. The word $x \in Q$, represents the shape of a picture whose longer side has length $|x|$, shorter side has length ${}_{\circ}|x|$ (the length of the maximal prefix over the symbol ${\circ}$), and when $x \notin {\circ}^*$, the letters $h,v$ indicate which side is the longer one.

\begin{definition}
Define $\textsc{NSpaceRev}_{Q}$ as the class of quasi-unary languages that can be recognized by a one-tape nondeterministic Turing machine in $|x|$ space and executing at most ${}_\circ|x|$ head reversals, for any input $x \in Q$.
\end{definition}

\begin{theorem}[\cite{BeGoLo09}]
A set $N \subset \Z_+^2$ is recognizable if and only if the corresponding quasi-unary language is in $\textsc{NSpaceRev}_{Q}$.
\end{theorem}

Our main interest is in recognizability of functions $f$ in the sublinear scheme, and the theorem simplifies as follows:

\begin{corollary}
A function $f : \Z_+ \to \Z_+$ with $f(n) = O(n)$ is recognizable if and only if the quasi-unary encodings of pairs $(m, f(m))$ can be recognized in linear time, executing at most $O(f(n))$ head reversals.
\end{corollary}

\section{Block gluing functions}
\label{sec:NoGap}

In this section define a class of SFTs that will be used to realize a large class of block gluing functions.
First, we define an auxiliary SFT called the \emph{boxes SFT}.

\begin{definition}
The alphabet of the boxes SFT is $\{0\} \cup \{0,1\}^{\{N,E,W,S\}}$. We think of $\vec v \in \Z^2$ as a geometric square cell, and as the alphabet we have empty symbols $0$, as well as nonzero symbols, which have \emph{inner borders} on any of the four sides, which are identified by the corresponding bit having value $1$. (A nonzero symbol with no borders is distinct from a zero symbol.) For a symbol $a \neq 0$, we write $a.N, a.E, a.W, a.S$ for the four bits corresponding to the borders.

By a \emph{rototranslation} of a configuration we mean a configuration in its dihedral orbit, i.e.\ a configuration obtained from it through a sequence of $90$-degree rotations and axial flips, where the bits $\{0,1\}^{\{N,E,W,S\}}$ are also rotated in the geometrically obvious way.

We impose the followings rules for each configuration and all its rototranslations:
\begin{itemize}
\item[(A)] If opposite corners of a $2 \times 2$ block are nonzero, then the entire block is nonzero.
\item[(B)] In a horizontal domino $ab$ with only nonzero symbols, a side border cannot end abruptly: $a.N = 1$, then $b.N = 1$, $a.E = 1$ or $b.W = 1$.
\item[(C)] In a horizontal domino $ab$ where $a \neq 0$ and $b = 0$, we must have $a.E = 1$.
\end{itemize}
Furthermore, in each $2 \times 2$ block $\begin{bmatrix} a & b \\ c & d \end{bmatrix}$ where none of the symbols are $0$, we require the following:
\begin{enumerate}
\item[(D)] If there is an inner border, then there must be a through border: if $a.E = 1$, then $a.S = b.S = 1$, $c.N = d.N = 1$, $a.E = c.E = 1$, or $b.W = d.W = 1$.
\item[(E)] A through border cannot receive orthogonal borders from the inside: if $a.S = b.S = 1$, then $a.E = b.W = 0$.
\end{enumerate}
\end{definition}

An example configuration in the boxes SFT is shown in Figure~\ref{fig:BoxesExample}.

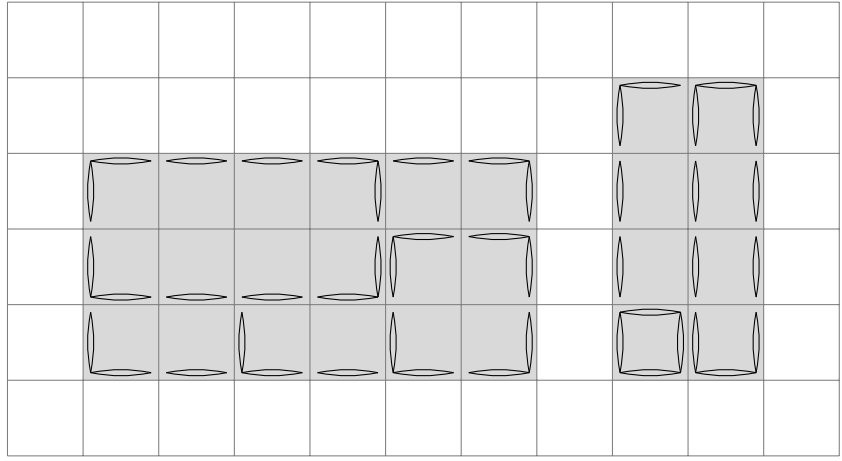
\begin{figure}
\begin{center}
\begin{tikzpicture}
\fill[black!15!white] (0,0) rectangle (6,3);
\fill[black!15!white] (7,0) rectangle (9, 4);
\draw[gray] (-1,-1) grid (10, 5);

\drawsides{0}{0}{0}{0}{1}{1};
\drawsides{0}{1}{0}{0}{1}{1};
\drawsides{0}{2}{0}{1}{1}{0};

\drawsides{1}{0}{0}{0}{0}{1};
\drawsides{1}{1}{0}{0}{0}{1};
\drawsides{1}{2}{0}{1}{0}{0};

\drawsides{2}{0}{0}{0}{1}{1};
\drawsides{2}{1}{0}{0}{0}{1};
\drawsides{2}{2}{0}{1}{0}{0};

\drawsides{3}{0}{0}{0}{0}{1};
\drawsides{3}{1}{1}{0}{0}{1};
\drawsides{3}{2}{1}{1}{0}{0};

\drawsides{4}{0}{0}{0}{1}{1};
\drawsides{4}{1}{0}{1}{1}{0};
\drawsides{4}{2}{0}{1}{0}{0};

\drawsides{5}{0}{1}{0}{0}{1};
\drawsides{5}{1}{1}{1}{0}{0};
\drawsides{5}{2}{1}{1}{0}{0};

\drawsides{7}{0}{1}{1}{1}{1};
\drawsides{8}{0}{1}{0}{1}{1};

\drawsides{7}{1}{0}{0}{1}{0};
\drawsides{8}{1}{1}{0}{1}{0};

\drawsides{7}{2}{0}{0}{1}{0};
\drawsides{8}{2}{1}{0}{1}{0};

\drawsides{7}{3}{0}{1}{1}{0};
\drawsides{8}{3}{1}{1}{1}{0};
\end{tikzpicture}
\end{center}
\caption{An example configuration in the boxes SFT. Two boxes are shaded and inner walls are drawn with lens shapes. Rule (A) corresponds to the fact the shaded areas are boxes and rule (C) corresponds to the fact the boxes have inner borders on their perimeter.}
\label{fig:BoxesExample}
\end{figure}

\begin{lemma}
\label{lem:Boxes}
In a globally valid finite pattern $p$, every position containing a nonzero symbol belongs to a unique rectangular region containing only nonzero symbols, such that immediate $\ell^\infty$-neighbors of the region inside the domain of $p$ are zero.
\end{lemma}

\begin{proof}
Suppose $p$ is a valid rectangular pattern and $p_{\vec v} \neq 0$. Then $\vec v$ at least belongs to the $1 \times 1$ rectangular region $R = \{\vec v\}$. Let $\llb a, b\rrb \times \llb c, d \rrb$ be a maximal rectangular region containing $\vec v$. If there is an immediate $\ell^\infty$-neighbor $\vec u$ of $R$ inside the domain of $p$ containing a nonzero symbol, then repeated application of rule A shows that the entire side containing $\vec u$ (or two sides, if it is a diagonal neighbor) contain nonzero symbols, so $R$ is not maximal. Being surrounded by $0$s, the region is obviously unique.
\end{proof}

A rectangular block of nonzero cells is called \emph{north terminating}, if some horizontal domino $ab$ on its top row satisfies $a.N = b.N = 1$ and $\max(a.E, b.W) = 1$ (i.e.\ there is a north border that receives an orthogonal inner border).
East, west and south terminating blocks are defined analogously.

\begin{figure}
\begin{center}
\begin{tikzpicture}
\fill[red!20!white] (0,0) rectangle (6,3);
\fill[blue!30!white] (7,0) rectangle (9, 4);
\draw[gray] (-1,-1) grid (10, 5);

\drawsides{0}{0}{0}{0}{1}{1};
\node () at (0.5,0.5) {$a$};
\drawsides{0}{1}{0}{0}{1}{1};
\node () at (0.5,1.5) {$a$};
\drawsides{0}{2}{0}{1}{1}{0};
\node () at (0.5,2.5) {$a$};

\drawsides{1}{0}{0}{0}{0}{1};
\node () at (1.5,0.5) {$a$};
\drawsides{1}{1}{0}{0}{0}{1};
\node () at (1.5,1.5) {$a$};
\drawsides{1}{2}{0}{1}{0}{0};
\node () at (1.5,2.5) {$b$};

\drawsides{2}{0}{0}{0}{1}{1};
\node () at (2.5,0.5) {$N$};
\drawsides{2}{1}{0}{0}{0}{1};
\node () at (2.5,1.5) {$a$};
\drawsides{2}{2}{0}{1}{0}{0};
\node () at (2.5,2.5) {$b$};

\drawsides{3}{0}{0}{0}{0}{1};
\node () at (3.5,0.5) {$N$};
\drawsides{3}{1}{1}{0}{0}{1};
\node () at (3.5,1.5) {$b$};
\drawsides{3}{2}{1}{1}{0}{0};
\node () at (3.5,2.5) {$b$};

\drawsides{4}{0}{0}{0}{1}{1};
\node () at (4.5,0.5) {$a$};
\drawsides{4}{1}{0}{1}{1}{0};
\node () at (4.5,1.5) {$a$};
\drawsides{4}{2}{0}{1}{0}{0};
\node () at (4.5,2.5) {$W$};

\drawsides{5}{0}{1}{0}{0}{1};
\node () at (5.5,0.5) {$b$};
\drawsides{5}{1}{1}{1}{0}{0};
\node () at (5.5,1.5) {$b$};
\drawsides{5}{2}{1}{1}{0}{0};
\node () at (5.5,2.5) {$W$};

\drawsides{7}{0}{1}{1}{1}{1};
\node () at (7.5,0.5) {$b$};
\drawsides{8}{0}{1}{0}{1}{1};
\node () at (8.5,0.5) {$b$};

\drawsides{7}{1}{0}{0}{1}{0};
\node () at (7.5,1.5) {$S$};
\drawsides{8}{1}{1}{0}{1}{0};
\node () at (8.5,1.5) {$b$};

\drawsides{7}{2}{0}{0}{1}{0};
\node () at (7.5,2.5) {$S$};
\drawsides{8}{2}{1}{0}{1}{0};
\node () at (8.5,2.5) {$b$};

\drawsides{7}{3}{0}{1}{1}{0};
\node () at (7.5,3.5) {$S$};
\drawsides{8}{3}{1}{1}{1}{0};
\node () at (8.5,3.5) {$b$};
\end{tikzpicture}
\end{center}
\caption{A configuration of the SFT from the proof of Theorem~\ref{thm:mainfunc}, overlaid on the example from Figure~\ref{fig:BoxesExample}. Here, $L$ is the domino language with allowed horizontal and vertical patterns $aa, ab, bb$, and any pattern with $\#$ except horizontal $a\#$.}
\label{fig:BoxesExample2}
\end{figure}
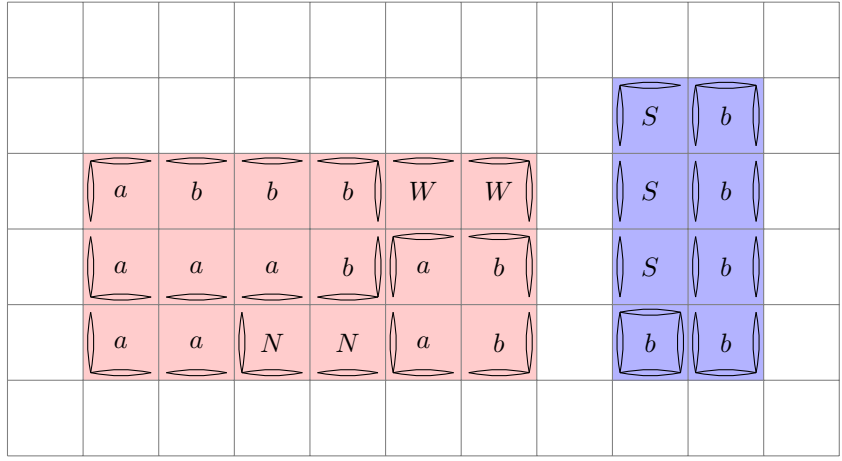

\begin{lemma}
\label{lem:terminating}
If a block is north terminating, then each cell on its top row has a north border, and the only locally valid upward continuation is with a row of zeroes.
\end{lemma}

\begin{proof}
By rule (E), the cells above the domino $ab$ cannot both be nonzero.
Repeated application of rule (A) gives that all cells immediately above the top row of the pattern are all zero.
Rule (C) implies that all cells on the top row have north borders.
\end{proof}

\begin{theorem}
\label{thm:mainfunc}
Let $L \subset A^{**}$ be a recognizable picture language. Then there exists an SFT such that
\[ B_H(n) \geq 2 S_H'(n) + 1, \]
\[ B_V(n) \geq 2 S_V'(n) + 1, \]
and
\[ B(n) = \max(3, 2 S'(n) + 1) \]
for all large enough $n$.
\end{theorem}

\begin{proof}
We assume $L$ is a domino language over an alphabet $A'$ given by horizontal and vertical matching rules, and we use special symbol $\# \notin A'$ outside the picture.

We preprocess $L$ by adding four symbols, which can only occur next to themselves, $E, N, W, S \notin A'$, and furthermore, $E$ (resp.\ $N$, $W$, $S$) cannot have a $\#$-symbol on its right (resp.\ above, left, below). Of course, the symbols do not change $|L|$, as these symbols do not appear in any valid pictures. Write $A = A' \cup \{E, N, W, S\}$. 

We consider the boxes SFT, and we overlay $A \times \{\textrm{red},\text{blue}\}$ on top of each nonzero symbol, so our SFT has zero symbols, and nonzero symbols which carry a symbol from the alphabet $A$, a color, and some set of borders (coming from the boxes SFT).
We require that two adjacent nonzero symbols have the same color. The effect of this is that each of the rectangular areas formed by nonzero symbols (see Lemma~\ref{lem:Boxes}) has a uniform color.

If there are no borders between two adjacent nonzero symbols on either side, then their $A$-symbols must match under the rules of $L$.
If a symbol has an inner border on some side, then its $A$-symbol must match $\#$ in that direction.
These rules (including the rules of the boxes SFT) define an SFT $X$.
An example configuration is shown in Figure~\ref{fig:BoxesExample2}.

Let $n \geq 4$. We first prove the lower bound $B_V(n) \geq 2S_V(n) + 1$.

Define an \emph{up cup} as a locally valid rectangular pattern of nonzero cells that has left borders on each cell of its leftmost row, right borders on its rightmost row, bottom borders on its bottom row, and no other borders.
Let $p$ be a globally valid pattern with domain $\llb m \rrb \times \llb n \rrb$ such that its top $i$ rows from two up cups, one on top of the other.
Consider the globally valid continuations of $p$ to the larger domain $\llb m \rrb \times \llb n+1 \rrb$.

First, observe that certainly such a pattern has a globally valid upward continuation of width $m$. If the upper up cup contains a valid bottom of a picture, then we can continue above with the next row of this picture. If this is the top row of the picture, then we have top borders on the new row. Alternatively, we can start a new picture by putting the bottom row of a valid picture of width $m$ above and surrounding it with inner borders (for example, we can copy the bottom row of the topmost up cup). We show that these are in fact the only possibilities.

First, since the nonzero symbols come from the boxes SFT, rule (A) implies that the only possibilities for the subset of the cells in $\llb m \rrb \times \{n\}$ that are nonzero are $\emptyset$ and $\llb m \rrb \times \{n\}$. The first case is not possible: by rule (C) it would require that $p$ has north borders on its top row.
Thus, we have $A$-symbols on each cell of the row $\llb m \rrb \times \{n\}$ above $p$.
Note that $p$ is west and east terminating, so by Lemma~\ref{lem:terminating} the only valid continuations to $\{-1, m\} \times \llb n+1 \rrb$ are with zeroes.
By rule (C) the cell $(0, n)$ must then have a left border, and $(m-1, n)$ must have a right border.

Suppose we have another border on row $n$, which is not on the top. It must be a bottom border, a right border or a left border at some $(i, n)$. In each of these cases, consider the at most two $2 \times 2$ boxes $q$ of $A$-symbols that contain it (so one or both of $\{i-1, i\} \times \{n-1, n\}$ and $\{i, i+1\} \times \{n-1, n\}$, depending on whether they fit in the nonzero area). At least one of these has an inner border, so by rule (D) of the boxes SFT, they must have a through border as well. Since cells in $\llb m \rrb \times \{n-1\}$ provide no inner borders, the only possibility is that we have bottom borders on row $n$ in $q$. Sweeping over the $2 \times 2$ patterns, we conclude that we have a lower border in all of the cells $\llb m \rrb \times \{n\}$ if there is even one inner border on this row. Furthermore, we cannot have any vertical inner borders, since by rule (E), a through border cannot receive orthogonal borders.

Now let us analyze the top borders. We may of course have no top borders on the entire row. If we have even one top border, then since we have no vertical inner borders on row $n$, and inner borders cannot end abruptly by rule (B), we must have a top border in all cells in $\llb m \rrb \times \{n\}$.

We conclude that the only possibilities for continuing $p$ upward by one row are that
\begin{itemize}
\item a bordered box of height $1$ (containing a valid picture) is added on the new row,
\item the topmost up cup continues into a bordered box,
\item a new up cup is formed on the top row, or
\item the up cup is continued.
\end{itemize}
(We may consider the first case a degenerate case of the second, where the new up cup is completed immediately.)
These are illustrated in Figure~\ref{fig:BoxesExtension}.

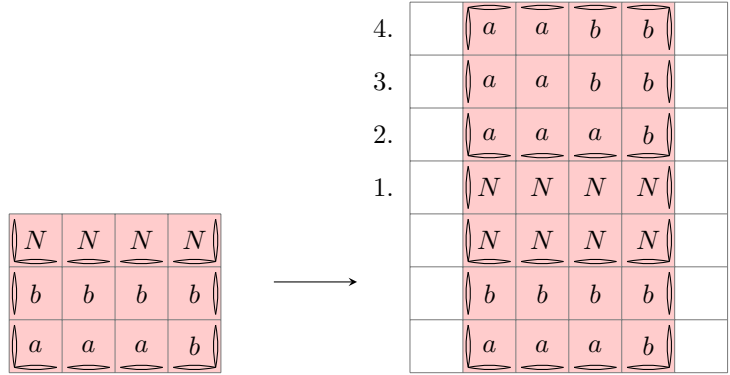
\begin{figure}
\begin{center}
\begin{tikzpicture}
\begin{scope}[scale=0.7]
\fill[red!20!white] (0,0) rectangle (4,3);
\draw[gray] (0,0) grid (4, 3);

\drawsides{0}{0}{0}{0}{1}{1};
\node () at (0.5,0.5) {$a$};
\drawsides{1}{0}{0}{0}{0}{1};
\node () at (1.5,0.5) {$a$};
\drawsides{2}{0}{0}{0}{0}{1};
\node () at (2.5,0.5) {$a$};
\drawsides{3}{0}{1}{0}{0}{1};
\node () at (3.5,0.5) {$b$};

\drawsides{0}{1}{0}{0}{1}{0};
\node () at (0.5,1.5) {$b$};
\drawsides{1}{1}{0}{0}{0}{0};
\node () at (1.5,1.5) {$b$};
\drawsides{2}{1}{0}{0}{0}{0};
\node () at (2.5,1.5) {$b$};
\drawsides{3}{1}{1}{0}{0}{0};
\node () at (3.5,1.5) {$b$};

\drawsides{0}{2}{0}{0}{1}{1};
\node () at (0.5,2.5) {$N$};
\drawsides{1}{2}{0}{0}{0}{1};
\node () at (1.5,2.5) {$N$};
\drawsides{2}{2}{0}{0}{0}{1};
\node () at (2.5,2.5) {$N$};
\drawsides{3}{2}{1}{0}{0}{1};
\node () at (3.5,2.5) {$N$};
\end{scope}

\draw[-stealth] (3.5,1.2) -- (4.6,1.2);

\begin{scope}[shift={(6,0)}, scale=0.7]
\fill[red!20!white] (0,0) rectangle (4,7);
\draw[gray] (-1,0) grid (5, 7);

\node () at (-1.5, 3.5) {1.};
\node () at (-1.5, 4.5) {2.};
\node () at (-1.5, 5.5) {3.};
\node () at (-1.5, 6.5) {4.};

\drawsides{0}{0}{0}{0}{1}{1};
\node () at (0.5,0.5) {$a$};
\drawsides{1}{0}{0}{0}{0}{1};
\node () at (1.5,0.5) {$a$};
\drawsides{2}{0}{0}{0}{0}{1};
\node () at (2.5,0.5) {$a$};
\drawsides{3}{0}{1}{0}{0}{1};
\node () at (3.5,0.5) {$b$};

\drawsides{0}{1}{0}{0}{1}{0};
\node () at (0.5,1.5) {$b$};
\drawsides{1}{1}{0}{0}{0}{0};
\node () at (1.5,1.5) {$b$};
\drawsides{2}{1}{0}{0}{0}{0};
\node () at (2.5,1.5) {$b$};
\drawsides{3}{1}{1}{0}{0}{0};
\node () at (3.5,1.5) {$b$};

\drawsides{0}{2}{0}{0}{1}{1};
\node () at (0.5,2.5) {$N$};
\drawsides{1}{2}{0}{0}{0}{1};
\node () at (1.5,2.5) {$N$};
\drawsides{2}{2}{0}{0}{0}{1};
\node () at (2.5,2.5) {$N$};
\drawsides{3}{2}{1}{0}{0}{1};
\node () at (3.5,2.5) {$N$};

\drawsides{0}{3}{0}{0}{1}{0};
\node () at (0.5,3.5) {$N$};
\drawsides{1}{3}{0}{0}{0}{0};
\node () at (1.5,3.5) {$N$};
\drawsides{2}{3}{0}{0}{0}{0};
\node () at (2.5,3.5) {$N$};
\drawsides{3}{3}{1}{0}{0}{0};
\node () at (3.5,3.5) {$N$};

\drawsides{0}{4}{0}{0}{1}{1};
\node () at (0.5,4.5) {$a$};
\drawsides{1}{4}{0}{0}{0}{1};
\node () at (1.5,4.5) {$a$};
\drawsides{2}{4}{0}{0}{0}{1};
\node () at (2.5,4.5) {$a$};
\drawsides{3}{4}{1}{0}{0}{1};
\node () at (3.5,4.5) {$b$};

\drawsides{0}{5}{0}{0}{1}{0};
\node () at (0.5,5.5) {$a$};
\drawsides{1}{5}{0}{0}{0}{0};
\node () at (1.5,5.5) {$a$};
\drawsides{2}{5}{0}{0}{0}{0};
\node () at (2.5,5.5) {$b$};
\drawsides{3}{5}{1}{0}{0}{0};
\node () at (3.5,5.5) {$b$};

\drawsides{0}{6}{0}{1}{1}{0};
\node () at (0.5,6.5) {$a$};
\drawsides{1}{6}{0}{1}{0}{0};
\node () at (1.5,6.5) {$a$};
\drawsides{2}{6}{0}{1}{0}{0};
\node () at (2.5,6.5) {$b$};
\drawsides{3}{6}{1}{1}{0}{0};
\node () at (3.5,6.5) {$b$};
\end{scope}
\end{tikzpicture}
\end{center}
\caption{A stack of two up cups for the SFT in Figure~\ref{fig:BoxesExample2}, and a possible upward continuation with four rows: 1. the up cup continues with another valid row, 2. another picture begins, 3. the picture continues, 4. the picture is completed into a valid picture of $L$. Columns of zeroes are forced on the left and right sides, since the initial pair of up cups is west and east terminating.}
\label{fig:BoxesExtension}
\end{figure}

We immediately conclude that if we start with an $m \times n$ pattern $p$ with only zeroes, except for up cups on the top two lines, then for the next $S_V(m)-1$ rows above, we must have symbols from $A$ with same color (as any bordered box must contain a picture from $L$). Furthermore, if the top row is not a valid bottom row of a picture in $L$, then we need at least $S_V(m)$ rows, since the initial up cup cannot be continued into a bordered box.

The same argument works below an $m \times n$ box, so by using a symmetric ``down cup'' on the bottom of a pattern with the opposite color and positioning it above $p$, we conclude that
\[ B_V(m) \geq 2S_V(m) + 1 \]
for all $m \geq 2$.
The claimed maximum over $m' \leq m$ is obtained by ignoring the $m - m'$ rightmost columns, and using the same argument.
Finally, the lower bound on $B_H(n)$ can be proved symmetrically, since the boxes SFT and $X$ are defined in a way that is obviously invariant under rototranslations.

Now we prove the upper bound $B(n) \leq \max(3, 2 S'(n)+1)$. Suppose by symmetry that we are gluing two globally valid $n \times n$ patterns on top of each other, with $\max(3, 2 S_V'(n') + 1)$ zero rows between them vertically (and any distance between them horizontally).

Consider the bottom pattern $p$, which we may suppose has domain $\llb n \rrb \times \llb n \rrb$. Our plan is to extend it upward into a full row of zeroes, which cuts the configuration in two independent pieces, and then symmetrically extend the pattern above. Figure~\ref{fig:ExtensionToCut} shows a typical example of this extension process.

The nonzero cells on the top row of $p$ split into maximal contiguous intervals. Each of these intervals $I$ is of some length $n' \leq n$. First, if $I$ is north terminating, then by Lemma~\ref{lem:terminating} the only possible upward continuation of $I$ is with zeroes. In particular, the continuation with zeroes must be globally valid, since we assumed $p$ is, and so we may put zeroes above $I$.

Suppose then $I$ does not contain adjacent top borders that receive an inner border.
If $I$ does not contain either of the top corners $(0, n-1), (n-1, n-1)$, then we position a bordered box of width $n'$ and height $S_V(n')$ taken from $L$ (independently of the other intervals) on top, with the same color as $I$.

If the top left corner is in $I$, but the top right corner is not, consider the maximal block of nonzero cells in $p$ that contains $I$.
If this block is west terminating, then, as before, we put a bordered box of width $n'$ and height $S_V(n')$ above it.
Otherwise, we put cells with the symbol $W$ on the row above $I$, with borders on the bottom, on the top, and on the right side of the rightmost cell.
If only the top right corner is in $I$, we do the extension symmetrically, with $E$ in place of $W$.

Finally, suppose that both corners are in $I$, so that is spans the entire row.
If the maximal nonzero block containing $I$ is both east and west terminating, we extend with an $n \times S_V(n)$ bordered box; if it is only west terminating, we extend with an $n \times 1$ box of $E$-symbols that is open on the right; if it is only east terminating, we extend with an $n \times 1$ box of $W$-symbols that is open on the left; and if it is neither east nor west terminating, we extend with and $n \times 1$ box of $E$-symbols that is open on left and right.


Finally, put zeroes on other columns above $p$, and zeroes above the bordered boxes constructed. We observe that the boxes are all of width at most $n$, so on the row $k = \max \{S_V(n') \;|\; n' \leq n\} + 1$ above $p$, we can place a full infinite row of zeroes, which we refer to as the \emph{cut row}.

We claim that the entire pattern under the cut row $k$ is globally valid. For this, first extend $p$ downward until an infinite row of zeroes with the same logic as we did above. Observe that this resulting pattern is locally valid.

We now attempt the symmetric extension process to the left of the pattern between the infinite rows of zeroes.
Thus, consider a maximal vertical interval $J$ of nonzero cells on the left border of the pattern.
If $J$ does not contain one of the corners $(0,0), (0,n-1)$ of $p$, the extension works as before.
Suppose that $J$ contains only the upper left corner $(0,n-1)$.
If $J$ is west terminating, then on top of the corner we have put either a zero, a fully enclosed box or an $E$-box that opens to the right, and in either case we may continue $J$ to the left with zeroes.
Otherwise, we may continue $J$ to the left with a fully enclosed box (or simply an infinite left-open box containing $W$).
The same argument works if $J$ contains only the lower left corner $(0,0)$, or both corners.

\begin{figure}
\begin{center}
\begin{tikzpicture}[scale=0.6]

    \fill[red!20!white] (0,3) rectangle (2,7);
    \fill[blue!30!white] (3,4) rectangle (9,12);
    \fill[red!20!white] (10,3) rectangle (13,7);
    \draw[gray] (-1,1) grid (14, 13);
    
    \draw[very thick](0,2) rectangle (12,6);

    \drawsides{0}{3}{0}{0}{1}{1}
    \node () at (0.5,3.5) {1};
    \drawsides{0}{4}{0}{0}{1}{1}
    \node () at (1.5,3.5) {0};
    \drawsides{0}{5}{0}{0}{1}{1}
    \node () at (0.5,4.5) {N};
    \drawsides{1}{3}{1}{0}{0}{1}
    \node () at (1.5,4.5) {N};
    \drawsides{1}{4}{1}{0}{0}{1}
    \node () at (0.5,5.5) {1};
    \drawsides{1}{5}{1}{0}{0}{1}
    \node () at (1.5,5.5) {0};

    \drawsides{3}{4}{0}{0}{1}{1}
    \node () at (3.5,4.5) {N};
    \drawsides{4}{4}{1}{0}{0}{1}
    \node () at (4.5,4.5) {N};
    \drawsides{5}{4}{1}{0}{0}{1}
    \node () at (5.5,4.5) {0};
    \drawsides{6}{4}{0}{0}{1}{1}
    \node () at (6.5,4.5) {1};
    \drawsides{7}{4}{0}{0}{0}{1}
    \node () at (7.5,4.5) {0};
    \drawsides{8}{4}{1}{0}{0}{1}
    \node () at (8.5,4.5) {0};

    \drawsides{3}{5}{0}{1}{1}{1}
    \node () at (3.5,5.5) {E};
    \drawsides{4}{5}{0}{1}{0}{1}
    \node () at (4.5,5.5) {E};
    \drawsides{5}{5}{0}{1}{0}{1}
    \node () at (5.5,5.5) {E};
    \drawsides{6}{5}{0}{0}{1}{0}
    \node () at (6.5,5.5) {2};
    \drawsides{7}{5}{0}{0}{0}{0}
    \node () at (7.5,5.5) {1};
    \drawsides{8}{5}{1}{0}{0}{0}
    \node () at (8.5,5.5) {0};

    \drawsides{10}{3}{0}{0}{1}{1}
    \drawsides{10}{4}{0}{0}{1}{0}
    \drawsides{10}{5}{0}{0}{1}{0}
    \drawsides{11}{3}{0}{0}{0}{1}
    \drawsides{11}{4}{0}{0}{0}{0}
    \drawsides{11}{5}{0}{0}{0}{0}
    \foreach \x in {10.5, 11.5}
        \foreach \y in {3.5,4.5,5.5} 
          \node () at (\x, \y) {$N$};

    \drawsides{0}{6}{0}{1}{1}{0}
    \node () at (0.5,6.5) {2};
    \drawsides{1}{6}{1}{1}{0}{0}
    \node () at (1.5,6.5) {1};

    \node () at (3.5, 6.5) {$1$};
    \node () at (4.5, 6.5) {$0$};
    \node () at (5.5, 6.5) {$0$};
    \node () at (6.5, 6.5) {$0$};
    \node () at (7.5, 6.5) {$0$};
    \node () at (8.5, 6.5) {$0$};
    \node () at (3.5, 7.5) {$2$};
    \node () at (4.5, 7.5) {$1$};
    \node () at (5.5, 7.5) {$0$};
    \node () at (6.5, 7.5) {$0$};
    \node () at (7.5, 7.5) {$0$};
    \node () at (8.5, 7.5) {$0$};
    \node () at (3.5, 8.5) {$2$};
    \node () at (4.5, 8.5) {$2$};
    \node () at (5.5, 8.5) {$1$};
    \node () at (6.5, 8.5) {$0$};
    \node () at (7.5, 8.5) {$0$};
    \node () at (8.5, 8.5) {$0$};
    \node () at (3.5, 9.5) {$2$};
    \node () at (4.5, 9.5) {$2$};
    \node () at (5.5, 9.5) {$2$};
    \node () at (6.5, 9.5) {$1$};
    \node () at (7.5, 9.5) {$0$};
    \node () at (8.5, 9.5) {$0$};
    \node () at (3.5, 10.5) {$2$};
    \node () at (4.5, 10.5) {$2$};
    \node () at (5.5, 10.5) {$2$};
    \node () at (6.5, 10.5) {$2$};
    \node () at (7.5, 10.5) {$1$};
    \node () at (8.5, 10.5) {$0$};
    \node () at (3.5, 11.5) {$2$};
    \node () at (4.5, 11.5) {$2$};
    \node () at (5.5, 11.5) {$2$};
    \node () at (6.5, 11.5) {$2$};
    \node () at (7.5, 11.5) {$2$};
    \node () at (8.5, 11.5) {$1$};

    \foreach \x in {3,...,8} {
        \drawsides{\x}{6}{0}{0}{0}{1}
        \drawsides{\x}{11}{0}{1}{0}{0}
    }
    \foreach \y in {6,...,11} {
        \drawsides{3}{\y}{0}{0}{1}{0}
        \drawsides{8}{\y}{1}{0}{0}{0}
    }

    \drawsides{10}{6}{0}{1}{1}{1}
    \node () at (10.5, 6.5) {$E$};
    \drawsides{11}{6}{0}{1}{0}{1}
    \node () at (11.5, 6.5) {$E$};

    \foreach \y in {3, 4, 5, 6} {
        \drawsides{12}{\y}{1}{0}{1}{0}
    }
    \drawsides{12}{3}{0}{0}{0}{1}
    \drawsides{12}{6}{0}{1}{0}{0}
    \node () at (12.5,3.5) {1};
    \node () at (12.5,4.5) {2};
    \node () at (12.5,5.5) {2};
    \node () at (12.5,6.5) {2};
    
    \fill[fill opacity = 0.3, color=gray!30!white] (0,2) rectangle (12,6);
\end{tikzpicture}
\end{center}
\caption{An example of how to extend a globally valid $12$-by-$4$ pattern $p$ upward to obtain a cut row, in the case of the language $L_{\textrm{idepi}}$ from Example~\ref{ex:IdEpi}. The domain of $p$ is marked with a rectangle and slightly shaded. On the left, we have a west terminating block, so we continue with zeroes on the left, and extend to a complete block above. On the right, we do not have an east terminating block, so we continue with an $E$-block above (we also show the extension to the right). The central picture is continued to a valid picture in $L_{\textrm{idepi}}$.}
\label{fig:ExtensionToCut}
\end{figure}
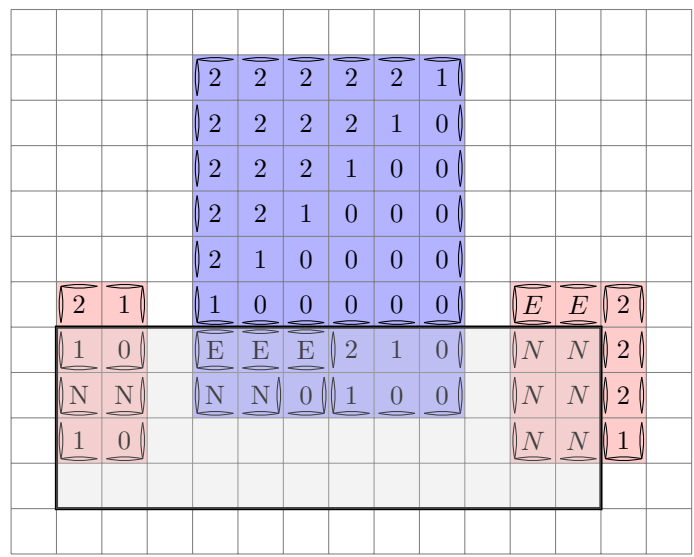

We then extend the pattern to the right symmetrically, and finally fill the remainder of the bottom half with zeroes.

Of course, above the cut row of all zeroes, we have an entirely symmetric situation. We conclude that we can also fill the top half in a valid way, since we have the same lower bound for the distance of the other $n \times n$ pattern from the zero row. This concludes the proof.
\end{proof}

\begin{theorem}
Let $f : \Z_+ \to \Z_+$ be a nondecreasing function whose epigraph is recognizable. Then there exists an SFT whose block gluing function satisfies $B(n) \approx B_V(n) \approx B_H(n) \approx f(n)$.
\end{theorem}

\begin{proof}
By Lemma~\ref{lem:sthing}, there is a picture language satisfying $S'_V(n) \approx f(n)$ and $S'_H(n) = 1$ for all $n$. The set of shapes of picture languages is closed under vertical squeezing by any rational number, so we can also find a picture language with $S'_V(n) \approx f(n)/2$ and $S'_H(n) = 1$. Then the SFT given by Theorem~\ref{thm:mainfunc} satisfies
\[ B(n) \geq B_V(n) \geq 2 \max\{S_V(n') \;|\; n' \leq n\} + 1 \approx f(n) \]
and
\[ B(n) \leq 2 \max\{S(n') \;|\; n' \leq n\} + 1 \approx f(n), \]
from which also $B_V(n) \approx B(n)$. By taking the Cartesian product of this SFT with its $90$-degree rotation, we obtain $B_V(n) \approx B_H(n) \approx B(n) \approx f(n)$.
\end{proof}

\begin{theorem}
Let $f : \Z_+ \to \Z_+$ be a nondecreasing recognizable function. Then there exists an SFT whose block gluing function satisfies $B(n) \approx B_V(n) \approx B_H(n) \approx f(n)$.
\end{theorem}

Combining with Proposition~\ref{prop:PosRatPow} gives:

\begin{proposition}
Let $r \in \Q$ be a positive rational number. Then there is an SFT whose block gluing distance satisfies $B(n) \approx B_V(n) \approx B_H(n) \approx n^r$.
\end{proposition}

\section{An aperiodic zero-entropy example}

We now prove that there are non-trivial (even aperiodic) linearly block gluing SFTs with zero entropy. The technical idea is somewhat similar to \cite{GaSa21}, where they apply a distortion operation to an SFT. We use instead a rigid wobbling method, which gives only entropy dimension $1$ rather than increasing the entropy.

\subsection{Wobbling an SFT}

Let $T$ be a set of Wang tiles with edge colors from $C$.
We define an SFT $W_T$ whose configurations are ``wobbly'' versions of those of $X_T$.
The construction could be extended to transform arbitrary SFTs into wobbly SFTs (or indeed arbitrary subshifts to wobbly subshifts), but we prefer to work with Wang tiles for simplicity's sake.

The alphabet of $W_T$ consists of all tiles of $T$ with any subset of edges decorated with \emph{fault lines}, plus the special tile $\vortex{0.25}$ called the \emph{vortex} with fault lines on all four sides.
It can be seen as a set of Wang tiles with edge colors $(C \cup \{\#\}) \times \{0,1\}$ with $1$ denoting a fault line and $\vortex{0.25}$ encoded as $({\#}, 1)^4$, but the adjacency rules are not those of Wang tiles.
Instead, the local rules of $W_T$ are as follows (see Figure~\ref{fig:wobble-rules}).
\begin{enumerate}
\item
  Fault lines must be matched with fault lines.
\item
  A $\vortex{0.25}$ cannot be adjacent to another $\vortex{0.25}$.
  The fault lines at the east, north, west and south edges of $\vortex{0.25}$ must continue to the north, east, south and west, respectively.
  In all other situations a fault line must continue in both directions and cannot cross or meet another fault line.
\item
  If an edge between adjacent tiles has no fault line, the $C$-colors of that edge must match.
\item
  A tile $t$ at $(i,j)$ has north color $(c, 1)$ with $c \in C$ (so a fault line) if and only if the tile at $(i+1,j+1)$ has south color $(c,1)$. Symmetrically, $t$ has east color $(c,1)$ with $c \in C$ if and only if the tile at $(i+1,j-1)$ has west color $(c,1)$.
\end{enumerate}

\begin{figure}[htp]
  \centering
  \begin{tikzpicture}

    \node at (-0.3,2.5) {1.};
    
    \fill[gray] (0.9,0) rectangle (1.1,1);
    \draw (0,0) rectangle (2,1);
    \draw (1,0) -- (1,1);

    \begin{scope}[shift={(2.5,0)}]
      \fill[gray] (0,0.9) rectangle (1,1.1);
      \draw (0,0) rectangle (1,2);
      \draw (0,1) -- (1,1);
    \end{scope}

    \begin{scope}[shift={(5,0)}]
      \node at (-0.3,2.5) {2.};
      
      \fill[gray] (0,1.9) -| (0.9,0) -| (1.1,0.9) -| (3,1.1) -| (2.1,3) -| (1.9,2.1) -| cycle;
      \drawvortex{(1,1)}
      \draw (0,0) grid (3,3);
      \foreach \x/\y in {1/0, 0/1, 1/2, 2/1}{
        \node at (\x+0.5,\y+0.5) {$\in T$};
      }
    \end{scope}

    \begin{scope}[shift={(0,-3)}]
      \node at (-0.3,2.5) {3.};
      
      \draw (0,0) rectangle (2,1);
      \draw (1,0) -- (1,1);
      \node at (0.8,0.5) {$c$};
      \node at (1.2,0.5) {$c$};

      \begin{scope}[shift={(2.5,0)}]
        \draw (0,0) rectangle (1,2);
        \draw (0,1) -- (1,1);
      \node at (0.5,0.8) {$c$};
      \node at (0.5,1.2) {$c$};
      \end{scope}
    \end{scope}

    \begin{scope}[shift={(5,-3)}]
      \node at (-0.3,2.5) {4.};
      
      \fill[gray] (0,0.9) rectangle (2,1.1);
      \draw (0,0) grid (2,2);
      \node at (1.5,1.2) {$c$};
      \node at (0.5,0.8) {$c$};

      \begin{scope}[shift={(2.5,0)}]
        \fill[gray] (0.9,0) rectangle (1.1,2);
        \draw (0,0) grid (2,2);
        \node at (0.8,1.5) {$c$};
        \node at (1.2,0.5) {$c$};
      \end{scope}
    \end{scope}

  \end{tikzpicture}
  \caption{The local rules of $W_T$.}
  \label{fig:wobble-rules}
\end{figure}
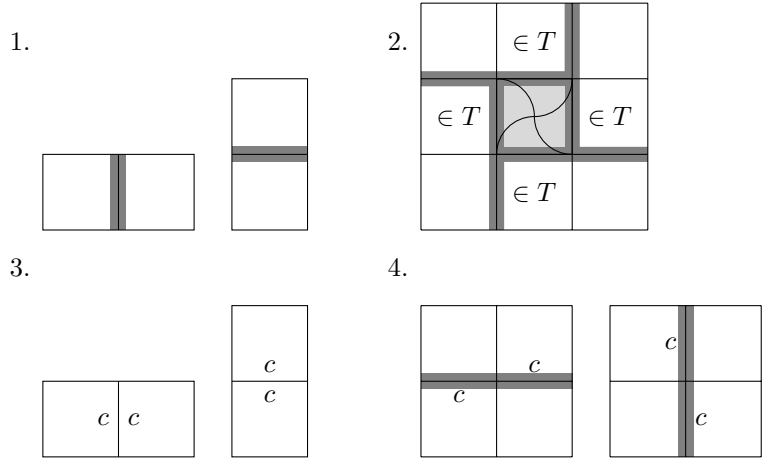

The intuition for these rules is the following.
We think of a configuration of $W_T$ as a tiling in $X_T$ to which some number of horizontal and vertical shears have been applied.
Each horizontal fault line represents a shear in which the north half of the tiling has been shifted one step to the east relative to the south half.
The color matching rules are shifted accordingly.
Symmetrically, each vertical fault line represents a shear in which the east half has been shifted one step to the south.
If fault lines would cross, an empty space is created, in which we place the vortex tile $\vortex{0.25}$.

\subsection{Constructing configurations of $W_T$}

We present a homeomorphism $f$ between $X_T \times \{0,1\}^\Z \times \{0,1\}^\Z$ and the subset of configurations of $W_T$ that don't have a $\vortex{0.25}$ at the origin.
This describes all configurations of $W_T$ up to a translation by $(1,0)$.
The map $f$ itself will be used in the proof of linear block gluing.

For $y, z \in \{0,1\}^\Z$, we define a function $g_{y,z} : \Z^2 \to \Z^2$ by $g_{y,z}(0,0) = (0,0)$, $g_{y,z}(a+1, b) = g_{y,z}(a,b) + (1, -y_a)$ and $g_{y,z}(a,b+1) = g_{y,z}(a,b) + (z_b, 1)$.
By induction, one obtains the formula
\begin{equation}
\label{eq:gyz-addition}
    g_{y,z}(a+i, b+j) = g_{y,z}(a,b) + (i + \sum z_{[b, b+j)}, j-\sum y_{[a, a+i)})
\end{equation}
for all $(a,b) \in \Z^2$ and $i, j \geq 0$.
It is easy to see that different choices of $y$ and $z$ result in different functions $g_{y,z}$.

Consider now the image of a single horizontal line: $H_b = g_{y,z}(\Z \times \{b\})$.
We have $g_{y,z}(a+1,b) \in g_{y,z}(a,b) + \{(1,0), (1,-1)\}$ for all $a \in \Z$, so $H_b$ looks like a horizontal line that sometimes tends to the southeast.
If $z_b = 0$, then $g_{y,z}(a,b+1) = g_{y,z}(a,b) + (0,1)$ for all $a \in \Z$, so in this case $H_{b+1} = H_b + (0,1)$.
In particular, $(i+1,j)$ is in the image of $g_{x,y}$ for each $(i,j) \in H_b$.
If $z_b = 1$, then $g_{y,z}(a,b+1) = g_{y,z}(a,b) + (1,1)$ for all $a \in \Z$, so in this case $H_{b+1} = H_b + (1,1)$.
For each coordinate $g_{z,y}(a,b)$ with $y_a = 1$, the neighbor $g_{y,z}(a,b) + (1,0)$ is not in the image of $g_{y,z}$.
This completely describes the image of $g_{y,z}$, which contains all of $\Z^2$ except $g_{y,z}(a,b) + (1,0)$ for $y_a = z_b = 1$.
We also see that if $\vec v \notin g_{y,z}(\Z^2)$, then each neighbor of $\vec v$ is in $g_{y,z}(\Z^2)$.
Finally, since the images $H_b$ are disjoint for different $b \in \Z$, the function $g_{y,z}$ is injective.
We have proved the following result.

\begin{lemma}
\label{lem:gyz-injective}
    For all $y, z \in \{0,1\}^\Z$, the function $g_{y,z} : \Z^2 \to \Z^2$ is injective, and $g_{y,z}(\Z^2) = \Z^2 \setminus \{g_{y,z}(a,b) + (1,0) \mid y_a = z_b = 1 \}$.
    The function $(y,z) \mapsto g_{y,z}$ is also injective.
\end{lemma}

Let $(x, y, z) \in X_T \times \{0,1\}^\Z \times \{0,1\}^\Z$ be arbitrary.
We define the image $x' = f(x, y, z)$ as follows.
For each $(a,b) \in \Z^2$, place the tile $x_{(a,b)}$ at $g_{y,z}(a,b)$.
Place a fault line on its east, north, west and south edge if and only if $y_a = 1$, $z_b = 1$, $y_{a-1} = 1$ and $z_{b-1} = 1$, respectively.
On each coordinate in $\Z^2 \setminus g_{y,z}(\Z^2)$, place a $\vortex{0.25}$.
As in \cite{GaSa21}, the general idea is to distort or ``wobble'' the configuration $x$, shearing it both in the horizontal and vertical direction, in order to move a given finite pattern of tiles into a predetermined position relative to the origin.
In \cite{GaSa21} shear maps can be applied independently in the horizontal and vertical directions, while in our construction the two directions are coupled.

\begin{lemma}
    The function $f$ is a homeomorphism from $X_T \times \{0,1\}^\Z \times \{0,1\}^\Z$ to the subset $W_T \setminus [\vortex{0.25}]_{(0,0)}$.
\end{lemma}

\begin{proof}
We first show that $x' = f(x,y,z) \in W_T\setminus [\vortex{0.25}]_{(0,0)}$.
Since $g_{y,z}(0,0) = (0,0)$, we have $x'_{(0,0)} \neq \vortex{0.25}$.

Consider an arbitrary $(a,b) \in \Z^2$ and the corresponding tile of $x'$ at $g_{y,z}(a,b)$, which contains $x_{(a,b)}$.
If $y_a = 0$, then the tile does not have a fault line on its east edge and the east neighbor $g_{y,z}(a+1,b) = g_{y,z}(a,b) + (1,0)$ contains $x_{(a+1,b)}$ without a fault line on its west edge.
These tiles satisfy rules 1 and 3.

If $y_a = 1$, then the tile has a fault line on its east edge.
Now the tile $x_{(a+1,b)}$ lies at $g_{y,z}(a+1, b) = g_{y,z}(a,b) + (1,-1)$ and has a fault line on its west edge, satisfying rule 4.
The east neighbor $g_{y,z}(a,b) + (1,0)$ is either $g_{y,z}(a+1,b+1)$ in the case $z_b = 0$, or a $\vortex{0.25}$ in the case $z_b = 1$, and in both cases has a fault line on its west edge.
Hence rules 1 and 4 are satisfied.

We have shown that the vertical edges of tiles in $x'$ follow rules 1, 3 and 4.
The case of horizontal edges is similar.
As for Rule 2, our description of $g_{y,z}(\Z^2)$ implies that no adjacent $\vortex{0.25}$-tiles occur.
Consider a $\vortex{0.25}$ at $g_{y,z}(a,b) + (1,0)$ with $y_a = z_b = 1$.
Its north neighbor is at $g_{y,z}(a,b) + (1,1) = g_{y,z}(a,b+1)$ and hence has a fault line on its east border, continuing the one on the east border of the $\vortex{0.25}$.
Symmetrically, we see that the fault lines on the other edges of the $\vortex{0.25}$ continue in the correct directions.

Consider then a tile $t$ at $g_{y,z}(a,b)$ with $y_a = 1$, with a fault line on its east border.
If $z_b = 1$, then $g_{y,z}(a,b) + (1,0)$ has a $\vortex{0.25}$-tile and the fault line is not required to continue to the north.
Otherwise, $g_{y,z}(a,b+1) = g_{y,z}(a,b) + (0,1)$ also has a fault line on its east border, continuing the one of $t$.
On the other hand, if $z_{b-1} = 1$ then $g_{y,z}(a,b) + (0,-1)$ has a $\vortex{0.25}$-tile, and otherwise it has another tile with a fault line on its east border; in both cases the fault line of $t$ continues to the south.
Hence fault lines on the east borders of tiles continue correctly.
The case of other fault lines is similar, so Rule 2 holds.
Hence $x' \in W_T$.

Injectivity of $f$ follows from Lemma~\ref{lem:gyz-injective}.
For surjectivity, consider and arbitrary $x' \in W_T\setminus [\vortex{0.25}]_{(0,0)}$.
Define $y, z \in \{0,1\}^\Z$ as follows: $y_k = 1$ if and only if the east border of the $k$th non-$\vortex{0.25}$ tile to the east of the origin has a fault line (the origin is the $0$th tile and the numbering extends to $k < 0$ as well), and similarly for $z$ and north borders of tiles to the north of the origin.
From the local rules of $W_T$ and the definition of $g_{y,z}$ we can deduce that the non-$\vortex{0.25}$ tiles of $x'$ lie exactly on the image of $g_{y,z}$.
The configuration $x \in X_T$ obtained by pulling these tiles back through $g_{y,z}$ satisfies $f(x,y,z) = x'$.

Finally, $f$ is continuous: denoting $D_k = \{(a,b) \in \Z^2 \mid |a| + |b| < k\}$, the pattern $f(x,y,z)|_{D_k}$ is determined by $x|_{[-2k+1, 2k]^2}$, $y|_{[-2k+1, 2k]}$ and $z|_{[-2k+1, 2k]}$.
\end{proof}

Consider the image $x' = f(x, y, z)$ of some $x \in X_T$ and $y, z \in \{0,1\}^\Z$.
Modifying the tiling $x$ does not change the structure of fault lines and $\vortex{0.25}$-tiles in $x'$, but only the contents of the non-$\vortex{0.25}$ tiles via the function $g_{y,z}$.
Suppose then that we replace $y_a = 0$ with a $1$.
This causes all tiles of $V_a = g_{y,z}(\{a\} \times \Z)$ to gain a fault line on their east edges and all tiles in $V_{a+1}$ to gain a fault line on their west edges.
If $a \geq 0$, then all tiles in $\bigcup_{i > a} V_i$ are shifted one step to the south, and if $a < 0$, then all tile in $\bigcup_{i \leq a} V_i$ are shifted one step to the north, with $\vortex{0.25}$-tiles filling the resulting gaps.
We can interpret this as introducing a new, infinitely long vertical fault line into $x'$.
Figure~\ref{fig:new-fault-line}, where the origin is marked with a dot, depicts this modification.
Conversely, replacing $y_a = 1$ with a $0$ will erase such a fault line, and modifying $z$ corresponds to creating and erasing infinite vertical fault lines and shifting parts of the configuration horizontally.

\begin{figure}[htp]
  \centering
  \begin{tikzpicture}[scale=0.4]

    \fill[gray] (-2,-1.1) rectangle (5,-0.9);
    \fill[gray] (-2,2.1) rectangle (5,1.9);

    \draw (-2,-3) grid (5,4);
    \fill (0.5,0.5) circle (0.2cm);
    \draw[dashed] (2,0.5) ellipse (0.3cm and 0.8cm);

    \begin{scope}[shift={(8,0)}]
      \fill[gray] (-2,-1.1) rectangle (5,-0.9);
      \fill[gray] (-2,2.1) rectangle (5,1.9);
      \fill[gray] (0.9,-3) rectangle (1.1,-1);
      \fill[gray] (1.9,-1) rectangle (2.1,2);
      \fill[gray] (2.9,2) rectangle (3.1,4);

      \draw (-2,-3) grid (5,4);
      \fill (0.5,0.5) circle (0.2cm);
    \end{scope}

    \begin{scope}[shift={(16,0)}]
      \fill[gray] (-2,-1.1) rectangle (2,-0.9);
      \fill[gray] (-2,2.1) rectangle (3,1.9);
      \fill[gray] (0.9,-3) rectangle (1.1,-1);
      \fill[gray] (1.9,-2) rectangle (2.1,2);
      \fill[gray] (2.9,1) rectangle (3.1,4);
      \fill[gray] (1,-2.1) rectangle (5,-1.9);
      \fill[gray] (2,0.9) rectangle (5,1.1);

      \drawvortex{(1,-2)}
      \drawvortex{(2,1)}

      \draw (-2,-3) grid (5,4);
      \fill (0.5,0.5) circle (0.2cm);
      \node at (3.5,1.5) {$\downarrow$};
      \node at (2.5,-1.5) {$\downarrow$};
    \end{scope}
    
  \end{tikzpicture}
  \caption{Adding a new vertical fault line by replacing $y_1 = 0$ with a $1$.}
  \label{fig:new-fault-line}
\end{figure}
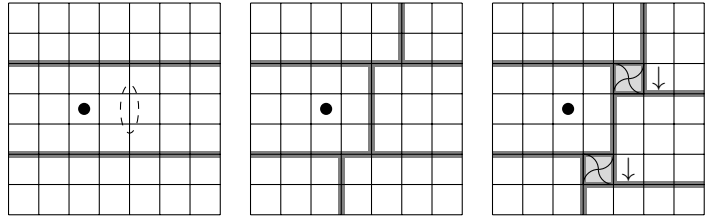

One can also think of $f(x, y, z)$ as the limit point of the process that starts with $x$ and proceeds by adding the fault lines dictated by $y$ and $z$ one by one, starting from the ones closest to the origin.
The process is depicted in Figure~\ref{fig:wobble-construction}.

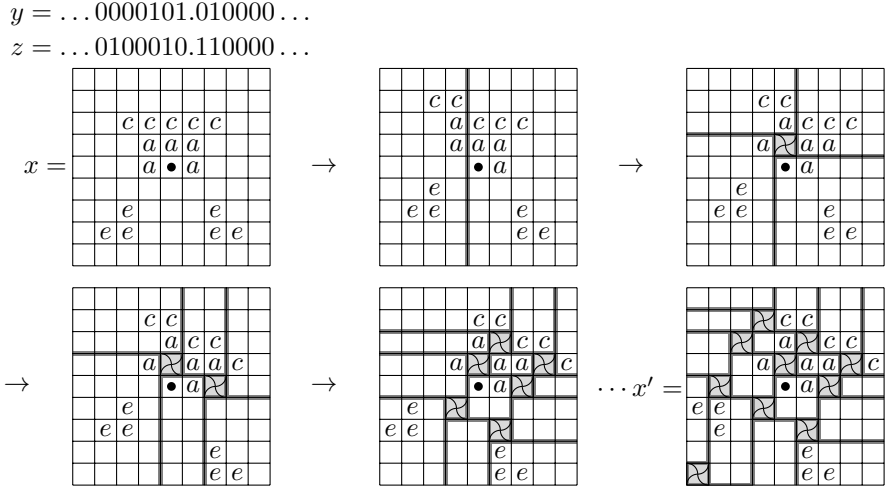
\begin{figure}[htp]
  \centering
  \begin{tikzpicture}[scale=0.29]

    \node at (0,7.5) {$y = \ldots 0 0 0 0 1 0 1 . 0 1 0 0 0 0 \ldots$};
    \node at (0,6) {$z = \ldots 0 1 0 0 0 1 0 . 1 1 0 0 0 0 \ldots$};

    \node [left] at (-3.5,0.5) {$x = {}$};
    \draw (-4,-4) grid (5,5);
    \fill (0.5,0.5) circle (0.2cm);
    \foreach \x/\y/\l in {
      1/0/a, 1/1/a, 0/1/a, -1/0/a, -1/1/a,
      -2/2/c, -1/2/c, 0/2/c, 1/2/c, 2/2/c,
      -2/-2/e, -2/-3/e, -3/-3/e, 2/-2/e, 2/-3/e, 3/-3/e
    }{
      \node at (\x+0.5, \y+0.5) {$\l$};
    }

    \begin{scope}[shift={(14,0)}]
      \node [left] at (-5.5,0.5) {$\rightarrow$};
      \fill[gray] (-0.1,-4) rectangle (0.1,5);
      \draw (-4,-4) grid (5,5);
      \fill (0.5,0.5) circle (0.2cm);
      \foreach \x/\y/\l in {
        1/0/a, 1/1/a, 0/1/a, -1/1/a, -1/2/a,
        -2/3/c, -1/3/c, 0/2/c, 1/2/c, 2/2/c,
        -2/-1/e, -2/-2/e, -3/-2/e, 2/-2/e, 2/-3/e, 3/-3/e
      }{
        \node at (\x+0.5, \y+0.5) {$\l$};
      }
    \end{scope}

    \begin{scope}[shift={(28,0)}]
      \node [left] at (-5.5,0.5) {$\rightarrow$};
      \fill[gray] (-4,1.9) rectangle (1,2.1);
      \fill[gray] (0,0.9) rectangle (5,1.1);
      \fill[gray] (-0.1,-4) rectangle (0.1,2);
      \fill[gray] (0.9,1) rectangle (1.1,5);
      \drawvortex{(0,1)}
      \draw (-4,-4) grid (5,5);
      \fill (0.5,0.5) circle (0.2cm);
      \foreach \x/\y/\l in {
        1/0/a, 2/1/a, 1/1/a, -1/1/a, 0/2/a,
        -1/3/c, 0/3/c, 1/2/c, 2/2/c, 3/2/c,
        -2/-1/e, -2/-2/e, -3/-2/e, 2/-2/e, 2/-3/e, 3/-3/e
      }{
        \node at (\x+0.5, \y+0.5) {$\l$};
      }
    \end{scope}

    \begin{scope}[shift={(0,-10)}]
      \node [left] at (-5.5,0.5) {$\rightarrow$};
      \fill[gray] (-4,1.9) rectangle (1,2.1);
      \fill[gray] (0,0.9) rectangle (3,1.1);
      \fill[gray] (2,-0.1) rectangle (5,0.1);
      \fill[gray] (-0.1,-4) rectangle (0.1,2);
      \fill[gray] (0.9,1) rectangle (1.1,5);
      \fill[gray] (1.9,-4) rectangle (2.1,1);
      \fill[gray] (2.9,0) rectangle (3.1,5);
      \drawvortex{(0,1)}
      \drawvortex{(2,0)}
      \draw (-4,-4) grid (5,5);
      \fill (0.5,0.5) circle (0.2cm);
      \foreach \x/\y/\l in {
        1/0/a, 2/1/a, 1/1/a, -1/1/a, 0/2/a,
        -1/3/c, 0/3/c, 1/2/c, 2/2/c, 3/1/c,
        -2/-1/e, -2/-2/e, -3/-2/e, 2/-3/e, 2/-4/e, 3/-4/e
      }{
        \node at (\x+0.5, \y+0.5) {$\l$};
      }
    \end{scope}

    \begin{scope}[shift={(14,-10)}]
      \node [left] at (-5.5,0.5) {$\rightarrow$};
      \fill[gray] (-4,1.9) rectangle (4,2.1);
      \fill[gray] (0,0.9) rectangle (5,1.1);
      \fill[gray] (2,-0.1) rectangle (5,0.1);
      \fill[gray] (-0.1,-1) rectangle (0.1,2);
      \fill[gray] (0.9,1) rectangle (1.1,3);
      \fill[gray] (1.9,-2) rectangle (2.1,1);
      \fill[gray] (2.9,0) rectangle (3.1,2);
      \fill[gray] (-4,2.9) rectangle (2,3.1);
      \fill[gray] (1.9,2) rectangle (2.1,5);
      \fill[gray] (3.9,1) rectangle (4.1,5);
      \fill[gray] (-4,0.1) rectangle (0,-0.1);
      \fill[gray] (-1,-0.9) rectangle (2,-1.1);
      \fill[gray] (1,-1.9) rectangle (5,-2.1);
      \fill[gray] (-1.1,-4) rectangle (-0.9,0);
      \fill[gray] (0.9,-4) rectangle (1.1,-1);
      \drawvortex{(0,1)}
      \drawvortex{(2,0)}
      \drawvortex{(1,2)}
      \drawvortex{(3,1)}
      \drawvortex{(-1,-1)}
      \drawvortex{(1,-2)}
      \draw (-4,-4) grid (5,5);
      \fill (0.5,0.5) circle (0.2cm);
      \foreach \x/\y/\l in {
        1/0/a, 2/1/a, 1/1/a, -1/1/a, 0/2/a,
        0/3/c, 1/3/c, 2/2/c, 3/2/c, 4/1/c,
        -3/-1/e, -3/-2/e, -4/-2/e, 1/-3/e, 1/-4/e, 2/-4/e
      }{
        \node at (\x+0.5, \y+0.5) {$\l$};
      }
    \end{scope}

    \begin{scope}[shift={(28,-10)}]
      \node [left] at (-3.75,0.5) {$\cdots x' =$};
      \fill[gray] (-2,1.9) rectangle (4,2.1);
      \fill[gray] (0,0.9) rectangle (5,1.1);
      \fill[gray] (2,-0.1) rectangle (5,0.1);
      \fill[gray] (-0.1,-1) rectangle (0.1,2);
      \fill[gray] (0.9,1) rectangle (1.1,3);
      \fill[gray] (1.9,-2) rectangle (2.1,1);
      \fill[gray] (2.9,0) rectangle (3.1,2);
      \fill[gray] (-4,2.9) rectangle (2,3.1);
      \fill[gray] (1.9,2) rectangle (2.1,5);
      \fill[gray] (3.9,1) rectangle (4.1,5);
      \fill[gray] (-3,0.1) rectangle (0,-0.1);
      \fill[gray] (-1,-0.9) rectangle (2,-1.1);
      \fill[gray] (1,-1.9) rectangle (5,-2.1);
      \fill[gray] (-1.1,-4) rectangle (-0.9,0);
      \fill[gray] (0.9,-4) rectangle (1.1,-1);
      \fill[gray] (-3.1,-4) rectangle (-2.9,1);
      \fill[gray] (-2.1,0) rectangle (-1.9,3);
      \fill[gray] (-1.1,2) rectangle (-0.9,4);
      \fill[gray] (-0.1,3) rectangle (0.1,5);
      \fill[gray] (-4,1.1) rectangle (-2,0.9);
      \fill[gray] (-4,3.9) rectangle (0,4.1);
      \fill[gray] (-4,-3.1) rectangle (-3,-2.9);
      \fill[gray] (-3,-4) rectangle (-1,-3.9);
      \drawvortex{(0,1)}
      \drawvortex{(2,0)}
      \drawvortex{(1,2)}
      \drawvortex{(3,1)}
      \drawvortex{(-1,-1)}
      \drawvortex{(1,-2)}
      \drawvortex{(-3,0)}
      \drawvortex{(-2,2)}
      \drawvortex{(-1,3)}
      \drawvortex{(-4,-4)}
      \draw (-4,-4) grid (5,5);
      \fill (0.5,0.5) circle (0.2cm);
      \foreach \x/\y/\l in {
        1/0/a, 2/1/a, 1/1/a, -1/1/a, 0/2/a,
        0/3/c, 1/3/c, 2/2/c, 3/2/c, 4/1/c,
        -3/-1/e, -3/-2/e, -4/-1/e, 1/-3/e, 1/-4/e, 2/-4/e
      }{
        \node at (\x+0.5, \y+0.5) {$\l$};
      }
    \end{scope}

  \end{tikzpicture}
  \caption{The construction of $x' = f(x,y,z)$ in stages.}
  \label{fig:wobble-construction}
\end{figure}

Finally, $f$ is approximately shift-equivariant in the following sense.

\begin{lemma}
\label{lem:f-shift}
    For all $x \in X_T$, $y, z \in \{0,1\}^\Z$ and $(a,b) \in \Z^2$ we have
    \[
    f(\sigma^{(a,b)} x, \sigma^a y , \sigma^b z) = \sigma^{g_{y,z}(a,b)}f(x,y,z).
    \]
\end{lemma}

\begin{proof}
    We first prove the claim for $(a,b) = (1,0)$.
    For this, denote $g = g_{y,z}$ and $h = g_{\sigma y, z}$.
    
    Let $(i,j) \in \Z^2$.
    Then $h(i,j) = g(i,j) + (0, y_0 - y_i)$: in the case $i \geq 0$ the calculation is
    \[
    h(i,j) - g(i,j) = (i, j-y_{[0, i)}) - (i, j-y_{[1, i+1)}) = (0, y_0 - y_i)
    \]
    and the $i < 0$ case is similar.
    Now we have
    \begin{align*}
    f(\sigma^{(1,0)} x, \sigma y, z)_{h(i,j)} = {} & (\sigma^{(0,1)} x)_{(i,j)} \\
    {} = {} & x_{(i+1, j)} \\
    {} = {} & f(x, y, z)_{g(i+1, j)} \\
    {} = {} & f(x,y,z)_{g(i,j) + (1, -y_i)} \\
    {} = {} & f(x,y,z)_{h(i,j) + (1, -y_0)} \\
    {} = {} & (\sigma^{g(1,0)} f(x, y, z))_{h(i,j)}.
    \end{align*}
    Since $(i,j)$ was arbitrary, this implies $f(\sigma^{(1,0)} x, \sigma y, z) = \sigma^{g(1,0)} f(x, y, z)$.
    
    The proof for $(a,b) = (0,1)$ is an essentially similar calculation.
    The general case then follows by composing such shifts.
\end{proof}

\subsection{Dynamical relation between $X_T$ and $W_T$}

We now prove that the wobbling construction preserves aperiodicity and low entropy dimension, and transforms a linearly block transitive shift into a linear block gluing one.

Although our wobbling is not quite the same as the distortion used in \cite{GaSa21}, the proof of \cite[Proposition~23]{GaSa21}, which shows that the distortion of an aperiodic SFTs stays aperiodic, applies essentially directly to our wobblings. However, in our case the proof is somewhat simpler to express, since we have a direct parametrization for the configurations.

\begin{lemma}
\label{lem:Aperiodic}
  If $X_T$ is aperiodic, then so is $W_T$.
\end{lemma}

\begin{proof}
  Suppose that a configuration $x' \in W_T$ satisfies $\sigma^{(i,j)}x' = x'$ for some $(i,j) \neq (0,0)$.
  We may assume $x'_{(0,0)} \neq \vortex{0.25}$.
  Then $(x,y,z) = f^{-1}(x')$ exists, $x'_{(i,j)} \neq \vortex{0.25}$, and thus $(i,j) = g_{y,z}(a,b)$ for some $(a,b) \in \Z^2$.
  By Lemma~\ref{lem:f-shift} we have
  \[
  f(x, y, z) = \sigma^{g_{y,z}(i,j)}f(x, y, z) = f(\sigma^{(a,b)} x, \sigma^a y, \sigma^b z),
  \]
  which implies $\sigma^{(a,b)} x = x$.

  If $x'$ is totally periodic, the set of periodicity vectors of $x'$ is syndetic, so there also exist $(a', b')$ with $a' \neq a$ and $b' \neq b$ such that $\sigma^{g_{y,z}(a', b')} x' = x'$.
  Then $\sigma^{(a',b')} x = x$, so $x$ is totally periodic.
\end{proof}

\begin{lemma}
\label{lem:ZeroEntropyDimension}
If the exact, upper or lower entropy dimension of $X_T$ is $d$, then the corresponding entropy dimension of $W_T$ is $\max(d, 1)$.
\end{lemma}

\begin{proof}
Of course, $W_T$ has a copy of every pattern of $X_T$ (when we have no fault lines). Thus, the entropy dimensions are no smaller than those of $X_T$. On the other hand, just looking at horizontal fault lines, we already see that all the three entropy dimensions must be at least $1$.

We now prove that if one of the entropy dimensions for $X_T$ is $d > 1$, then we have the corresponding bound for $W_T$. Consider all configurations $x \in X_T$ and $y, z \in \{0,1\}^\Z$. The region $[0, k]^2 \subset D_{2k+2}$ of $f(x, y, z)$ is determined by $x|_{[-2k-1, 2k+2]^2}$, $y|_{[-2k-1, 2k+2]}$ and $z|_{[-2k-1, 2k+2]}$.

  Hence, the number of patterns $P \in \lang_{k+1}(W_T)$ with $P_{(0,0)} \neq \vortex{0.25}$ is at most $2^{8k+4} |\lang_{2k+2}(X_T)|$.
  Each pattern $Q \in \lang_k(W_T)$ occurs in such a $P$ with at most three other patterns, so that $N_k(W_T)  
  \leq 2^{8k+6} N_{2k+2}(X_T)$. 

For the upper entropy dimension we have
\begin{align*}
\bar{D}(W_T) &= \limsup_{k \rightarrow \infty} \frac{\log(\log(N_k(W_T))}{\log k} \\
&\leq \limsup_{k \rightarrow \infty} \frac{\log(\log(2^{8k+6} N_{2k+2}(X_T))}{\log k} \\
&\leq \limsup_{k \rightarrow \infty} \frac{\log(\log(2^{8k+6}) + \log(N_{2k+2}(X_T)))}{\log k} \\
&\leq \limsup_{k \rightarrow \infty} \frac{\log(O(k) + \log(N_{2k+2}(X_T)))}{\log k} \\
&\leq \limsup_{k \rightarrow \infty} \frac{\max(\log(k), \log(\log(N_{2k+2}(X_T)))) + O(1)}{\log k} \\
&= \limsup_{k \rightarrow \infty} \max\left(1, \frac{\log(\log(N_{2k+2}(X_T)))}{\log k}\right) \\
&\leq \max\left(1, \limsup_{k \rightarrow \infty} \frac{\log(\log(N_{k}(X_T)^3))}{\log k}\right) \\
&\leq \max\left(1, \limsup_{k \rightarrow \infty} \frac{\log(3\log(N_{k}(X_T)))}{\log k}\right) \\
&\leq \max\left(1, \limsup_{k \rightarrow \infty} \frac{\log(\log(N_{k}(X_T))) + O(1)}{\log k}\right) \\
&\leq \max(1, \bar{D}(X_T))
\end{align*}
and similarly for other entropy dimensions.
\end{proof}

\begin{lemma}
\label{lem:LinearBG}
  If $X_T$ is weakly linearly block transitive with constant $K$, then $W_T$ has linear block gluing with constant $46K+30$.
\end{lemma}


\begin{proof}
  Let $n > 0$ and $P_1, P_2 \in \lang_n(W_T)$ be arbitrary, and let $(i,j) \in \Z^2$.
  We investigate whether there exists a configuration of $W_T$ containing both patterns with offset $(i,j)$.
  
  Let $x'_1, x'_2 \in W_T$ and $s_1, s_2 \in \{0,1\}$ be such that $P_i$ occurs in $x'_i$ at $(s_i, 0)$ and $x'_i$ does not contain a $\vortex{0.25}$ at the origin.
  Then there exist $x_i \in X_T$ and $y_i, z_i \in \{0,1\}^\Z$ such that $f(x_i, y_i, z_i) = x'_i$ for $i = 1, 2$.

  Suppose that $X_T$ is weakly $(g,g)$-block transitive with $g(n) = Kn$.
  Then there exists a configuration $x \in X_T$ such that $Q_1 = x_1|_{[-n, n] \times [0, 2n-1]}$ occurs in $x$ at $(0,0)$ and $Q_2 = x_2|_{[-n, n] \times [0, 2n-1]}$ occurs in $x$ at $(i,j) - (2N, N) + \vec w$ for some vector $\vec w$ with $\|\vec w\|_\infty \leq K(2n+1)$.
  Here $N > 0$ is an auxiliary parameter, and each $Q_i$ contains every symbol of the configuration $x_i$ that can end up in $[0, n] \times [0, n-1]$ under the map $g_{y_i, z_i}$ regardless of $y_i$ and $z_i$.


  We now add fault lines to the configuration $x$ in order to transform the occurrences of $Q_1$ and $Q_2$ into patterns that contain $P_1$ and $P_2$.
  Since at most $2n$ horizontal and $2n$ vertical fault lines can touch an $n \times n$ pattern of $W_T$, this can be achieved by adding at most $4n$ horizontal and $4n$ vertical fault lines, each of which enters $Q_1$ or $Q_2$.
  The fault lines can be added independently, in the sense that each fault line in $f(x,y,z)$ corresponds to a vertical of horizontal line in $x$, and the line borders $\vec v \in \Z^2$ if and only if the fault line borders $g_{z,y}(\vec v)$.
  Hence, it suffices to verify that $Q_1$ and $Q_2$ are separated horizontally and vertically, with $Q_2$ to the northwest of $Q_1$, and for this,
  \begin{equation}
  \label{eq:constr1}
  \begin{cases}
  i < 2N - 2Kn - K - n, \\
  j \geq N + 2Kn + K + 2n,
  \end{cases}
  \end{equation}
  is enough.
  
  Each of the at most $2n$ horizontal fault line added to $Q_1$ shifts the occurrence of $Q_2$ by one step to the east, and each of the at most $2n$ vertical fault line added to $Q_2$ shifts $Q_2$ to the north.
  The other fault lines do not affect the patterns' positioning.
  Hence, in the resulting configuration $x'$ the pattern $P_1$ occurs at $(s_1, 0)$ and $P_2$ occurs at $(i-2N+a, j-N+b)$ with $a, b \in [-K(2n+1), K(2n+1)+2n+1]$.

  We now modify the configuration $x'$ by adding new fault lines between the occurrences of $P_1$ and $P_2$, with the goal of moving $P_2$ to the position $(s_1+i, j)$.
  This can be achieved by adding $2N-a+s_1 \leq 2N + 2Kn + K + 1$ horizontal fault lines and $N-b \leq N + 2Kn + K$ vertical fault lines between the two patterns, as long as the movements are in the correct directions, i.e.\ the minimum values of $2N-a+s_1$ and $N-b$ are nonnegative.
  This holds if
  \begin{equation}
      \label{eq:FaultLinesPos}
          N \geq 2Kn+K+2n+1.
  \end{equation}

  We also need enough room between the original patterns $Q_1$ and $Q_2$ to add these new fault lines; note that the fault lines already added to produce $P_1$ and $P_2$ all intersect $Q_1$ and $Q_2$, so they can be ignored in this computation.
  The vertical separation between the top of $Q_1$ and the bottom of $Q_2$ is at least $j-N-K(2n+1)-(2n-1)$.
  Hence, this quantity should be larger than $2N + 2Kn + K + 1$, or equivalently,
  \begin{equation}
      \label{eq:constr2}
      j > 3N + 4Kn + K + 2n + 1.
  \end{equation}
  Note that this constraint is stronger than the second inequality of \eqref{eq:constr1}.
  
  By equation \eqref{eq:constr1}, the east border of $Q_2$ is to the west of $Q_1$, so their horizontal separation is at least $-n - (i - 2N + 2Kn + K) = -i + 2N - 2Kn - K - n$.
  Hence, it suffices to have $-i + 2N - 2Kn - K - n > N + 2Kn + K$, or equivalently,
  \begin{equation}
      \label{eq:constr3}
      i < N - 4Kn - 2K - n.
  \end{equation}

  We have shown that for a given value of the parameter $N$, the gluing at offset $(i,j)$ is possible if equations \eqref{eq:constr1}, \eqref{eq:FaultLinesPos}, \eqref{eq:constr2} and \eqref{eq:constr3} hold.
  Consider the parameter $N = N_0 = 2Kn+K+2n+1$, for which \eqref{eq:FaultLinesPos} holds with equality.
  In this case the gluing is possible for all $i < -2Kn-K+n+1$ and $j > 3N_0 + 4Kn+K+2n+1 = 10Kn + 4K + 8n + 4$.
  
  Consider then $N = N_1 = 14Kn + 6K + 9n + 4$.
  Now the gluing is possible for all $i < N_1 - 4Kn - 2K - n = 10Kn + 4K + 8n + 4$ and $j > 3N_1 + 4Kn + K + 2n + 1 = 46Kn + 19K + 29n + 13$.
  
  By rotational symmetry of the construction of $W_T$, whenever gluing is possible with offset $(i,j)$ using this argument, it is also possible with offsets $(-j,i)$, $(-i,-j)$ and $(j,-i)$.
  Hence, gluing is possible with offset $(i,j)$ whenever $|i|, |j| > 46Kn + 19K + 29n + 13$, so $W_T$ is linearly block gluing with constant $46K + 30$.
\end{proof}



\subsection{Linear block gluing}

We now apply the results of the previous section to existing subshifts.

\begin{theorem}
\label{thm:1548}
There exists a weakly linearly block gluing aperiodic SFT, which has entropy dimension 1.
\end{theorem}

Specifically we prove this with the gluing constant 1548.

\begin{proof}
In \cite[Section~2.4]{GaSa21}, an aperiodic SFT $X_T$ is constructed which is $32\textrm{id}$-net gluing. This means that for all $p, q \in \lang_n(X)$ there exists $\vec v \in \Z^2$ such that for all $\vec w \in 33n(\Z^2 \setminus \{(0,0)\}) + \vec v$, we have $X \cap [p] \cap \sigma^{-\vec w}[q] \neq \emptyset$. Clearly this implies weakly linear block transitivity with gluing constant $33$.
By Lemma~\ref{lem:LinearBG}, the SFT $W_T$ has linear block gluing with constant $1548$, and by Lemma~\ref{lem:Aperiodic} it is aperiodic.

The SFT $X_T$ consists of the Robinson tile set \cite{Ro71}, on top of which a new ``synchronization layer'' is superimposed.
The Robinson layer has entropy dimension zero: to specify a globally valid $n \times n$ block, it suffices to fix the phases of the first $\log_2 n + O(1)$ layers, and for each we have 4 possibilities, giving a complexity of $O(n^2)$.

The synchronization layer decorates each long arrow with one of two possible decorations.
A sequence of consecutive arrows must have identical decorations, and each cross determines the decorations of the arrows that originate from it.
In locally valid square patterns of $X_T$, the Robinson layer completely determines the synchronization layer apart from sequences of arrows originating from outside the pattern.
Hence the upper entropy dimension of $X_T$ is at most 1, and by Lemma~\ref{lem:ZeroEntropyDimension}, $W_T$ has entropy dimension 1.
\end{proof}

We also give an alternative construction from \cite{Ol08}, for which the existence of a linear gluing constant follows from the general theory, although we do not obtain a concrete one without an additional calculation.

Recall that a subshift is \emph{linearly recurrent} if every globally valid $n \times n$ pattern appears in every globally valid $kn \times kn$ pattern for some constant $k \geq 1$. Clearly a linearly recurrent subshift is weakly linearly block transitive.

Let $A$ be an alphabet and $\tau : A \to A^{r \times r}$ any substitution. We can apply $\tau$ to a rectangular pattern 
in an obvious way, by replacing each symbol by its $\tau$-image. The corresponding \emph{substitutive} subshift $X_\tau \subset A^{\Z^2}$ is defined by forbidding all patterns that do not occur in $\tau^n(a)$ for any $n \geq 0$ and $a \in A$. The following is classical:

\begin{lemma}
\label{lem:LinearlyRecurrent}
Suppose $\tau$ is primitive, meaning for some $n \geq 1$, every symbol $b \in A$ appears in the pattern $\tau^n(a)$ for all $a \in A$. Then $X_\tau$ is linearly recurrent.
\end{lemma}

\begin{proof}
Every $2 \times 2$ pattern that appears in $X_\tau$ appears in some $\tau^k(b)$, where $b \in A$ and $k$ is bounded from above. By primitivity, every such pattern appears in $\tau^{k+m}(a)$ for every $a \in A$ and $m \geq n$. Hence there exists $M$ such that every $2 \times 2$ pattern appears in $\tau^M(a)$ for every $a \in A$. Thus, every $2 \times 2$ pattern $q$ appears in $X_{\tau}$ in every $2r^M \times 2r^M$ block, and for $m \geq 0$, $\tau^m(q)$ appears in every $2r^{m+M} \times 2r^{m+M}$-block.

Consider a globally valid $\ell \times \ell$ pattern $p$. Since it is not forbidden, it appears in some $\tau^k(b)$.
Inside this pattern, $p$ appears in a subpattern of the form $\tau^m(q)$ for some $2 \times 2$ pattern $q$ and $m = \lceil \log_r \ell \rceil$.
Then $p$ appears in every $2r^{m+M} \times 2r^{m+M}$ block.
Here $2r^{m+M} < 2r^{M+1} \ell$ is linear in $\ell$.
\end{proof}

\begin{proof}[Alternative proof of Theorem~~\ref{thm:1548}]
It suffices to show that there is a tile set $T$ such that $X_T$ is aperiodic, has entropy dimension at most 1, and is weakly linearly block transitive. In \cite{Ol08a}, Ollinger constructs an intrinsically substitutive aperiodic SFT.

Specifically, it is $2$-by-$2$ substitutive with an injective deterministic substitution. Such subshifts are easily seen to have zero entropy dimension: the number of patterns of shape $n$-by-$n$ is indeed linear in $n^2$, thus $W_T$ has entropy dimension $1$ by Lemma~\ref{lem:ZeroEntropyDimension}.

Since $X_T$ is aperiodic, by Lemma~\ref{lem:Aperiodic} $W_T$ is as well. It is shown in \cite{Ba09} that the substitution is primitive (specifically this follows from Lemma~1.33 and Theorem 1.35). A primitive substitutive subshift is linearly recurrent (Lemma~\ref{lem:LinearlyRecurrent}), thus weakly linearly block transitive. Now Lemma~\ref{lem:LinearBG} shows that $W_T$ is linearly block-gluing.
\end{proof}

\bibliographystyle{plain}
\bibliography{bib}{}

\end{document}